\documentclass[12pt]{amsart}

\usepackage[colorlinks=true, pdfstartview=FitV, linkcolor=blue,
citecolor=blue, urlcolor=blue]{hyperref}

\usepackage{mathtools}
\usepackage{amsfonts}
\usepackage{amsmath}
\usepackage{amssymb} 
\usepackage{amsthm}

\usepackage{times}
\usepackage[T1]{fontenc}
\usepackage{enumitem}
\usepackage{setspace}
\usepackage{microtype}
\usepackage{cite}

\allowdisplaybreaks

\usepackage{tikz}
\usepackage{graphicx}
\usetikzlibrary{decorations.pathreplacing,calligraphy}

\usepackage[margin=1.2in]{geometry}

\newtheorem{theorem}{Theorem}[section]
\newtheorem{lemma}[theorem]{Lemma}
\newtheorem{prop}[theorem]{Proposition}
\newtheorem{corollary}[theorem]{Corollary}

\theoremstyle{definition}
\newtheorem{definition}[theorem]{Definition}
\newtheorem{example}[theorem]{Example}

\theoremstyle{remark}
\newtheorem{remark}[theorem]{Remark}

\numberwithin{equation}{section}

\newcommand{\R}{\mathbb R}
\newcommand{\N}{\mathbb N}

\newcommand{\rn}{{{\mathbb R}^n}}

\DeclareMathOperator{\supp}{supp}

\DeclareMathOperator*{\essinf}{ess\,inf}
\DeclareMathOperator*{\esssup}{ess\,sup}

\newcommand{\pp}{{p(\cdot)}}
\newcommand{\cpp}{{p'(\cdot)}}
\newcommand{\Lp}{L^{p(\cdot)}}
\newcommand{\Pp}{\mathcal P}

\newcommand{\qq}{{q(\cdot)}}

\def\Xint#1{\mathchoice
   {\XXint\displaystyle\textstyle{#1}}%
   {\XXint\textstyle\scriptstyle{#1}}%
   {\XXint\scriptstyle\scriptscriptstyle{#1}}%
   {\XXint\scriptscriptstyle\scriptscriptstyle{#1}}%
   \!\int}
\def\XXint#1#2#3{{\setbox0=\hbox{$#1{#2#3}{\int}$}
     \vcenter{\hbox{$#2#3$}}\kern-.5\wd0}}

\def\dashint{\Xint-}

\newcommand{\E}{\mathcal{E}}
\newcommand{\Kalconst}[2]{[ \hspace{0.3mm} #1 \hspace{0.3mm} ]_{K_0^\alpha(#2)}}
\newcommand{\Kconst}[2]{[ \hspace{0.3mm} #1 \hspace{0.3mm} ]_{K_0(#2)}}

\begin{document}

\title[Necessary conditions for the boundedness of fractional operators]
{Necessary conditions for the boundedness of fractional operators on variable Lebesgue spaces}

\author{David Cruz-Uribe, OFS}
\address{Dept. of Mathematics \\
University of Alabama \\
 Tuscaloosa, AL 35487, USA}
\email{dcruzuribe@ua.edu}

\author{Troy Roberts}
\address{Dept. of Mathematics \& Statistics\\
Washington University in St. Louis\\
 One Brookings Drive\\  St. Louis, MO, 63130,  USA}
\email{r.troy@wustl.edu}



\thanks{This paper is a revised version of the masters thesis written by the second author under the direction of the first.  The first author was partially supported by a Simons Foundation Travel Support for Mathematicians Grant and is currently supported by NSF grant DMS-2349550. }

\date{July 11, 2024}

\begin{abstract}
  In this paper we prove necessary conditions for the boundedness of fractional operators on the variable Lebesgue spaces.  More precisely, we find necessary conditions on an exponent function $\pp$ for a fractional maximal operator $M_\alpha$ or a non-degenerate fractional singular integral operator $T_\alpha$, $0 \leq \alpha < n$, to satisfy weak $(\pp,\qq)$ inequalities or strong $(\pp,\qq)$ inequalities, with $\qq$ being defined pointwise almost everywhere by 
  \[ \frac{1}{p(x)} - \frac{1}{q(x)} = \frac{\alpha}{n}. \]
 We first prove preliminary results linking fractional averaging operators and the $K_0^\alpha$ condition, a qualitative condition on $\pp$ related to the norms of characteristic functions of cubes, and show some useful implications of the $K_0^\alpha$ condition. We then show that if $M_\alpha$ satisfies weak $(\pp,\qq)$ inequalities, then $\pp \in K_0^\alpha(\R^n)$. We use this to prove that if $M_\alpha$ satisfies strong $(\pp,\qq)$ inequalities, then $p_->1$. Finally, we prove a powerful pointwise estimate for $T_\alpha$ that relates $T_\alpha$ to $M_\alpha$ along a carefully chosen family of cubes. This allows us to prove necessary conditions for fractional singular integral operators similar to those for fractional maximal operators.
\end{abstract}

\maketitle
\section{Introduction}
\indent \indent The variable Lebesgue spaces, $L^{\pp}(\R^n)$, are generalizations of the classical Lebesgue spaces, $L^p$, where the constant exponent $p$ is replaced with a variable exponent function $\pp$. These spaces have been extensively studied for more than thirty years:  for a detailed history, see the introduction of \cite{VLS}. They have been studied both for their intrinsic interest as Banach function spaces and also due to their use in the treatment of partial differential equations with non-standard growth conditions (see \cite{RuzickaElectrofluids,Mingione2006,Chan1997,Fan2003,Nuortio2010}). 

\indent A common problem in harmonic analysis is to determine when operators are bounded between Banach spaces. In the case of the variable Lebesgue spaces, this leads to the natural question of how classical operators, such as maximal functions and singular integral operators, behave on these spaces. To put our results in context, we recall some results on the classical and weighted Lebesgue spaces.

It is well-known that the Hardy-Littlewood maximal operator,
\[ Mf(x) = \sup_{Q \ni x} \dashint_{Q} |f(y)| \, dy, \]
where the supremum is taken over cubes containing $x$ with sides parallel to the coordinate axis, is bounded from $L^p(\R^n) \to L^p(\R^n)$ if and only if $1 < p \leq \infty$ (see \cite{GraBook}). 
The fractional maximal operator,
\[ M_\alpha f(x) = \sup_{Q \ni x} |Q|^{\frac{\alpha}{n}} \dashint_Q |f(y)| \, dy, \hspace{10mm} 0 < \alpha < n, \]
was introduced by Muckenhoupt and Wheeden~\cite{MuckenhouptWheeden1974};  they proved that for $1 < p \leq \tfrac{n}{\alpha}$,  $M_\alpha : L^p(\R^n) \to L^q(\R^n)$, with $q$ being defined by 
\begin{equation} \label{eqn:sobolev}
 \frac{1}{p} - \frac{1}{q} = \frac{\alpha}{n}. 
 \end{equation}
%

In the weighted Lebesgue spaces, the situation is more complicated.  The maximal operator satisfies $M : L^p(v) \rightarrow L^{p,\infty}(u)$, $1<p<\infty$, if and only if the pair of weights $(u,v)$ satisfies the Muckenhoupt $A_p$ condition.  Similarly, the fractional maximal operator satisfies  $M_\alpha : L^p(v) \rightarrow L^{q,\infty}(u)$, $1 \leq  p < \tfrac{n}{\alpha}$, if and only if  $(u,v)$ satisfies the two-weight fractional Muckenhoupt condition, $A_{p,q}$.  Moreover, in the one-weight case (i.e., when $u=v$) these conditions are sufficient for the  corresponding strong-type inequalities.   However, in the two-weight case, they are not sufficient for the strong-type inequalities.  (For definitions and proofs, see~\cite{GraBook,MuckenhouptWheeden1974,MR3642364}.)  

We are also interested in norm inequalities for singular integrals.  Calder\'on-Zygmund operators (CZOs) are singular integral operators whose kernels satisfy certain decay and smoothness conditions.  They  satisfy $L^p$ bounds for $1 < p < \infty$, but are not bounded on $L^1$ or $L^\infty$ (see \cite{GraBook}).  Fractional CZOs are defined similarly, but satisfy $L^p\rightarrow L^q$ estimates for $1<p<\frac{n}{\alpha}$.  (See~\cite{MR3470665,MR3688149}.  A precise definition is given in Section \ref{OperInt}.) Important examples of CZOs are the Hilbert transform and its higher dimensional analogues, the Riesz transforms. 

In the weighted Lebesgue spaces, in the one-weight case the Muckenhoupt $A_p$ condition is sufficient for a CZO to satisfy both weak and strong $(p,p)$ inequalities, $1<p<\infty$, and this condition is also necessary if the operator satisfies a weak non-degeneracy condition introduced by Stein~\cite{SteinBook} which is satisfied, for instance, by the Riesz transforms.  Similarly, fractional CZOs satisfy weak and strong $(p,q)$ inequalities, $1<p<\frac{n}{\alpha}$, if and only if the weight satisfies the Muckenhoupt $A_{p,q}$ condition.  In the two-weight case, however, these conditions are necessary but not sufficient.  (For the one-weight results for CZOs, see~\cite{SteinBook}; for fractional CZOs these results are implicit in~\cite{MuckenhouptWheeden1974,MR3642364}.  For the two-weight results, see~\cite{MR3470665,MR4319616}.)

\medskip

The boundedness properties of these operators have been extensively studied on variable Lebesgue spaces, but the question of  the "right" necessary and sufficient conditions for them to be bounded are not completely understood. To state the known results, we first give some preliminary definitions.  
Let $\mathcal{P}(\R^n)$ be the set of all Lebesgue measurable functions $\pp : \R^n \to [1,\infty]$. Given $\pp \in \mathcal{P}(\R^n)$, we write
    \begin{gather*}
    p_- = \essinf_{x \in \R^n} p(x), \hspace{8mm} p_+ = \esssup_{x \in \R^n} p(x).
\end{gather*}  
Hereafter, given $\pp \in \mathcal{P}(\R^n)$ such that $p_+ \leq \tfrac{n}{\alpha}$ for some $\alpha$, $0 \leq \alpha < n$, we define $q(\cdot) \in \mathcal{P}(\R^n)$ by
\begin{equation} \label{q dot defn}
    \frac{1}{p(x)} - \frac{1}{q(x)} = \frac{\alpha}{n}, \hspace{10mm} x \in \R^n.
\end{equation}
This definition generalizes~\eqref{eqn:sobolev} to the variable exponent setting.

It is known that the exponent function $\pp$ must have some kind of regularity for these operators to be bounded.  Early on it was proved that log-H\"older continuity is sufficient for the Hardy-Littlewood maximal operator to be bounded on $\Lp$,  assuming $p_- > 1$, and for weak-type bounds when $p_- = 1$ (see \cite{DieningMaxBddConstExp,Nekvinda,MR1976842,VLS}). It is also sufficient for the fractional maximal operator to satisfy $M_\alpha : \Lp \rightarrow L^{\qq}$, where $\qq$ is defined by~\eqref{q dot defn}
(see~\cite{MR2210118}).  
Log-H\"older continuity is also sufficient for CZOs to be bounded on $L^\pp$ if $1<p_-\leq p_+<\infty$ (see \cite{MR2009242,MR2210118}), and for the Riesz potentials (and so, implicitly, fractional CZOs) to map $L^\pp$ to $L^\qq$ if $1<p_-\leq p_+ <\frac{n}{\alpha}$ (see~\cite{MR2210118}). 

However, log-H\"older continuity is not necessary for any of these operators to be bounded. For the Hardy-Littlewood maximal operator this was shown by Lerner in \cite{lerner2005some}, who constructed an example of a discontinuous exponent function $\pp$ such that $M$ is bounded on $\Lp$.  This example also yields examples for the other operators:  this follows from the theory of Rubio de Francia extrapolation in the scale of variable Lebesgue spaces (see~\cite{MR2210118,VLS}).  (This fact is part of the folklore of variable Lebesgue spaces, but does not seem to be in the literature.  We sketch the short proof in  Remark~\ref{remark:extrapol} below.)

Diening~\cite{NeccAndSuffDiening} gave a necessary and sufficient condition for the maximal operator to be bounded  that is difficult to check in practice, but has many important theoretical consequences. More recently, Lerner~\cite{lerner2023boundedness} gave a different necessary and sufficient condition that is also difficult to check.   CZOs are bounded on $\Lp$, provided $1<p_-\leq p_+<\infty$, if and only if the maximal operator is.  The sufficiency follows from extrapolation (again see~\cite{MR2210118,VLS}); the necessity was proved by Rutsky~\cite{MR3977081,MR3580165} in the setting of Banach lattices. (We note in passing that it would be interesting to have a direct proof of this result for variable Lebesgue spaces.)  

Comparable necessary and sufficient results are not known for fractional maximal operators, Riesz potentials, or fractional CZOs. The goal of this paper, therefore, is to begin to address this question by establishing 
necessary conditions for fractional maximal operators and fractional CZOs to be bounded on the variable Lebesgue spaces. 
We first consider the bounds on the oscillation of the exponents.  It has been shown that if the Hardy-Littlewood maximal operator is bounded on $L^{\pp}(\R^n)$, then $p_- > 1$ \cite[Theorem 3.19]{VLS}. Our first result extends this to the fractional maximal operators. 
\begin{theorem} \label{Maximal Bounded Exponent Theorem}
    Fix $\alpha$, $0 \leq \alpha < n$, and let $\pp \in \mathcal{P}(\R^n)$.  Define $\qq \in \mathcal{P}(\R^n)$ by \eqref{q dot defn}.  If $M_\alpha : L^{\pp}(\R^n) \to L^{\qq}(\R^n)$, then $p_- > 1$. 
\end{theorem}

\begin{remark}
The proof of this result for the maximal operator (i.e., when $\alpha=0$) is relatively straightforward.  However, this proof fails to work when $\alpha>0$; consequently, a different and significantly more difficult proof is required.
\end{remark}

Likewise, it is known that if all $n$ of the Riesz transforms are bounded on $L^{\pp}(\R^n)$, then $1 < p_- \leq p_+ < \infty$ \cite[Theorem 5.42]{VLS}. We prove that for $0 \leq \alpha < n$, it suffices to have that a single non-degenerate fractional CZO (see Definitions~\ref{CZ operator} and~\ref{non-degenerate kernel}) is bounded to get the same conclusion.  

\begin{theorem} \label{SIO Bounded Exponent Theorem}
    Fix $0 \leq \alpha < n$, and let $\pp \in \mathcal{P}(\R^n)$.  Define $\qq \in \mathcal{P}(\R^n)$ by \eqref{q dot defn}.
    Let $T_\alpha$ be a fractional Calder\'on-Zygmund operator with a non-degenerate kernel. If  $T_\alpha : L^{\pp}(\R^n) \to L^{\qq}(\R^n)$, then $1 < p_- \leq p_+ < \frac{n}{\alpha}$.
\end{theorem}

\begin{remark}
We note that Theorem~\ref{SIO Bounded Exponent Theorem} is new even for classical CZOs (i.e., when $\alpha=0$), since the conclusion follows if, for example, a single Riesz transform is bounded.  
\end{remark}

An important regularity condition on exponent functions for the boundedness of the maximal operator and CZOs is the $K_0(\R^n)$ condition. 
    Given $\pp \in \mathcal{P}(\R^n)$, we say that $\pp \in K_0(\R^n)$ if
    \[ \Kconst{\pp}{\R^n} = \sup_{Q} |Q|^{-1} \lVert \chi_Q \rVert_{\pp} \lVert \chi_Q \rVert_{p'(\cdot)} < \infty, \]
    where the supremum is taken over cubes $Q \subset \R^n$ with sides parallel to the coordinate axes. 
In \cite[Corollary 4.50, Theorem 5.42]{VLS} it was shown that if the Hardy-Littlewood maximal operator or all $n$ of the Riesz transforms satisfy weak $(\pp,\pp)$ bounds, then $\pp \in K_0(\R^n)$.   This condition is analogous to the Muckenhoupt $A_p$ condition, mentioned above; to see this, note that we can rewrite this condition as follows:  $w\in A_p$ if
\[ \sup_Q |Q|^{-1} \|w^{\frac{1}{p}}\chi_Q\|_p \|w^{-\frac{1}{p}}\chi_Q\|_{p'} < \infty.   \]
However, there is a crucial difference:  while the $A_p$ condition is necessary and sufficient for the boundedness of maximal operators and non-degenerate CZOs, the $K_0$ condition is not sufficient for the maximal operator to be bounded on $\Lp$ (see~\cite[Example~4.5.1]{VLS}).

Our next two results extend this fractional operators.  Given $0 \leq \alpha <n$, we say that $\pp \in K_0^\alpha (\R^n)$ if
    \[ \Kalconst{\pp}{\R^n} = \sup_{Q} |Q|^{\frac{\alpha}{n}-1} \lVert \chi_Q \rVert_{\pp} \lVert \chi_Q \rVert_{p'(\cdot)} < \infty.  \]

\begin{theorem} \label{Maximal K_0^alpha theorem}
    Fix $\alpha$, $0 \leq \alpha < n$, and let $\pp \in \mathcal{P}(\R^n)$. Define $q(\cdot) \in \mathcal{P}(\R^n)$ by \eqref{q dot defn}. If  $M_\alpha : L^{\pp}(\R^n) \to L^{q(\cdot),\infty}(\R^n)$, then $\pp \in K_0^\alpha(\R^n)$.
\end{theorem}
\begin{theorem} \label{SIO K_0^alpha Theorem}
Fix $\alpha$, $0 \leq \alpha < n$, and let $\pp \in \mathcal{P}(\R^n)$. Define $q(\cdot) \in \mathcal{P}(\R^n)$ by \eqref{q dot defn}.  Let $T_\alpha$ be a fractional Calder\'on-Zygmund operator with a non-degenerate kernel. If $T_\alpha : L^{p(\cdot)}(\R^n) \to L^{\qq, \infty}(\R^n)$,  then $p(\cdot) \in K_0^\alpha(\R^n)$. 
\end{theorem}

The $K_0^\alpha$ condition is analogous to the $A_{p,q}$ condition of Muckenhoupt and Wheeden, but again, while the $A_{p,q}$ condition is necessary and sufficient for fractional maximal operators and fractional CZOs to be bounded on weighted spaces, this is not the case in variable Lebesgue spaces and we construct explicit examples of exponents in $K_0^\alpha(\R^n)$ to show this.

\medskip

The remainder of the paper is organized as follows. In Section~\ref{Preliminaries} we give some preliminary definitions and results on the variable Lebesgue spaces and the fractional operators we are interested in.  In Section~\ref{Averagin Operators and the K_0^alpha Condition} we define the fractional averaging operators and show that their boundedness is characterized by the $K_0^\alpha(\R^n)$ condition. This result linking averaging operators and the $K_0^\alpha$ condition is our main tool in proving Theorems~\ref{Maximal K_0^alpha theorem} and~\ref{SIO K_0^alpha Theorem}. We conclude this section with some  consequences of the $K_0^\alpha(\R^n)$ condition needed in our proofs. In Section~\ref{Fractional Maximal Operators} we first prove Theorem~\ref{Maximal K_0^alpha theorem}; then, as a corollary, we show that  $M_\alpha : L^{\pp} \to L^{\qq}$ implies $\pp \in K_0^\alpha(\R^n)$. Using this fact,  we next prove Theorem \ref{Maximal Bounded Exponent Theorem}. We motivate the proof, which is very technical, by looking at the constant exponent case first.   In Section~\ref{Fractional Singular Integral Operators} we prove Theorems~\ref{SIO K_0^alpha Theorem} and~\ref{SIO Bounded Exponent Theorem}. To do so, we first prove a pointwise estimate that relates fractional CZOs to the fractional maximal operator along a carefully chosen family of cubes. This allows us to adapt much of the machinery developed in Section~\ref{Fractional Maximal Operators} to the proof for fractional CZOs.  We note that the proofs for the fractional maximal operator could be simplified, but at the cost of repeating the more complicated construction we actually use in this section.  Finally, in Section~\ref{Necessary but Not Sufficient} we show that the $K_0^\alpha$ condition is not sufficient for $M_\alpha$  to be bounded. We also prove some additional results that show the relationship between the $K_0$ and $K_0^\alpha$ conditions and some of their properties. 

\medskip

Throughout this paper we will use the following notation. The variable $n$ will always denote the dimension of Euclidean space, $\R^n$. Unless we specify otherwise, cubes in $\R^n$ will  be assumed to have their sides parallel to the coordinate axes.  $Q(x,r)$ denotes the cube with center $x \in \R^n$ and side length $2r$. We refer to the parameter $r$ as the radius of $Q(x,r)$. Likewise, $B(x,r)$ denotes the ball with center $x$ and radius $r$. Constants may shift from line to line, and in many cases we will display dependence on a parameter $p$ by writing $C(p)$. We will write $A\lesssim B$ if there exists $c>0$ such that $A\leq cB$, and $A\approx B$ if $A\lesssim B$ and $B\lesssim A$. Given a  measurable function $f$ and a measurable set $E$, $0 < |E| < \infty$, we define the integral average of $f$ on $E$ by
\[ \dashint_{E} f(x) \, dx = \frac{1}{|E|} \int_{E} f(x) \, dx. \]
 Finally, for variable Lebesgue spaces we will use the notation from~\cite{VLS}.

%
%
\section{Preliminaries} \label{Preliminaries}

In this section we define the variable Lebesgue spaces and the operators we will be considering in subsequent sections.

\subsection{The Variable Lebesgue Spaces}
We begin with the definition of the variable Lebesgue spaces and give some of their fundamental properties.   Results are given without proof, and can be found in \cite{VLS,Diening}. 
%
%
%
\begin{definition}
    Given a set $\Omega\subset \rn$, let $\mathcal{P}(\Omega)$ be the set of all Lebesgue measurable functions $\pp : \Omega \to [1,\infty]$. The elements of $\mathcal{P}(\Omega)$ are called exponent functions. 
\end{definition}
Given $\pp \in \mathcal{P}(\Omega)$ and a set $E \subset \Omega$, let 
\begin{gather*}
    p_-(E) = \essinf_{x \in E} p(x), \hspace{8mm} p_+(E) = \esssup_{x \in E} p(x), \\
    \Omega_\infty^{\pp} = \{ x \in \Omega : p(x) = \infty \}, \hspace{8mm} \Omega_1^{\pp} = \{x \in \Omega : p(x) = 1\}, \\
    \Omega_*^{\pp} = \{x \in \Omega : 1 < p(x) < \infty \}.
\end{gather*}  
If the context is clear, we will  write $p_- = p_-(\Omega)$ and $p_+ = p_+(\Omega)$. We will also drop the superscript and write $\Omega_\infty,\Omega_1$, and $\Omega_*$.
\begin{definition}
    Given $\Omega$, $\pp \in \mathcal{P}(\Omega)$ and a Lebesgue measurable function $f$, define the modular associated with $\pp$ by
    \[ \rho_{\pp, \Omega}(f) = \int_{\Omega \setminus \Omega_\infty} |f(x)|^{p(x)} \, dx + \lVert f \rVert_{L^\infty(\Omega_\infty)}. \]
\end{definition}
When it is unambiguous, we may omit one or both subscripts on $\rho$. The modular allows us to define the variable Lebesgue spaces.
\begin{definition}
    Given $\Omega$ and $\pp \in \mathcal{P}(\Omega)$, define $L^{\pp}(\Omega)$ to be the set of Lebesgue measurable functions $f$ such that $\rho(f/\lambda) < \infty$ for some $\lambda > 0$. 
\end{definition}
The modular also allows us to define a norm on $L^{\pp}(\Omega)$ via a similar approach as the Luxemburg norm on Orlicz spaces. As with Orlicz spaces, this forms a Banach space. 
\begin{definition}
    Given $\Omega$ and $\pp \in \mathcal{P}(\Omega)$, if $f$ is a measurable function, define 
    \[ \lVert f \rVert_{L^{\pp}(\Omega)} = \inf \{ \lambda > 0 : \rho_{\pp,\Omega}(f/\lambda) \leq 1 \}. \]   
\end{definition}
\begin{theorem}[\cite{VLS}, Theorem 2.71]
    Given $\Omega$ and $\pp \in \mathcal{P}(\Omega)$, $\|\cdot\|_{L^{\pp}(\Omega)}$ is a norm, and $\Lp(\Omega)$ is a Banach space.
\end{theorem}
As is often done in the classical Lebesgue spaces, when the domain is understood we will write $\lVert f \rVert_{\pp} = \lVert  f \rVert_{L^{\pp}(\Omega)}$. Similar to the classical Lebesgue spaces, we  define the conjugate exponent function of $\pp$ pointwise: for $x \in \Omega$,
\begin{gather} \label{conjugate exp}
    \frac{1}{p(x)} + \frac{1}{p'(x)} = 1,
\end{gather}  
where we use the convention $\tfrac1\infty = 0$. 
H\"older's inequality is satisfied in the variable Lebesgue spaces up to a constant. Useful generalizations of H\"older's inequality hold as well. 
\begin{theorem}[\cite{VLS}, Theorem 2.26] \label{Holders}
    Given $\Omega$ and $\pp \in \mathcal{P}(\Omega)$, for all $f \in L^{\pp}(\Omega)$ and $g \in L^{p'(\cdot)}(\Omega)$, $fg \in L^1(\Omega)$ and 
    \[ \int_\Omega |f(x) g(x)| \, dx \leq K_{\pp} \lVert f \rVert_{\pp} \lVert g \rVert_{p'(\cdot)}, \]
    where
    \[ K_{\pp} = \left(\frac{1}{p_-} - \frac{1}{p_+} + 1 \right) \lVert \chi_{\Omega_*} \rVert_\infty + \lVert \chi_{\Omega_\infty} \rVert_\infty + \lVert \chi_{\Omega_1} \rVert_\infty \leq 4. \]
\end{theorem}
\begin{remark}
    Since
    \[ \frac{1}{(p')_-} - \frac{1}{(p')_+} 
    = \frac{1}{(p_+)'} - \frac{1}{(p_-)'} 
    = \left(1 - \frac{1}{p_+}\right) - \left(1 - \frac{1}{p_-}\right) = \frac{1}{p_-} - \frac{1}{p_+}, \]
    $K_{\pp} = K_{p'(\cdot)}$ for all $\pp \in \mathcal{P}(\R^n)$.
\end{remark}
\begin{corollary}[\cite{VLS}, Corollary 2.28] \label{gen Holder}
    Given $\Omega$ and exponent functions $r(\cdot), q(\cdot) \in \mathcal{P}(\Omega)$, define $\pp \in \mathcal{P}(\Omega)$ by 
    \[ \frac{1}{p(x)} = \frac{1}{q(x)} + \frac{1}{r(x)}. \]
    Then there exists a constant $K$, $K \leq 5$, such that for all $f \in L^{\qq}(\Omega)$ and $g \in L^{r(\cdot)}(\Omega)$, $fg \in L^{\pp}(\Omega)$ and 
    \[ \lVert fg \rVert_{\pp} \leq K \lVert f \rVert_{\qq} \lVert g \rVert_{r(\cdot)}. \]
\end{corollary}
\begin{corollary}[\cite{VLS}, Corollary 2.30]
    Given $\Omega$, suppose $p_1(\cdot), p_2(\cdot), \dots, p_k(\cdot) \in \mathcal{P}(\Omega)$ are  exponent functions that satisfy
    \[ \sum_{i=1}^k \frac{1}{p_i(x)} = 1, \hspace{1cm} x \in \Omega. \]
    Then there exists a constant $C$, depending on $k$ and the $p_i(\cdot)$, such that for all $f_i \in L^{p_i(\cdot)}(\Omega)$, $1 \leq i \leq k$,
    \[ \int_\Omega |f_1(x) f_2(x) \dots f_k(x)| \, dx \leq C \lVert f_1 \rVert_{p_1(\cdot)} \lVert f_2 \rVert_{p_2(\cdot)} \cdots \lVert f_k \rVert_{p_k(\cdot)}. \]
\end{corollary}
Next we state some important duality relationships in the variable Lebesgue spaces.
\begin{theorem}[\cite{VLS}, Theorem 2.34] \label{associate norm}
    Given $\Omega$, $\pp \in \mathcal{P}(\Omega)$, and a measurable function $f$, 
    \[ k_{\pp} \lVert f \rVert_{\pp} \leq \sup \int_\Omega f(x) g(x) \, dx \leq K_{\pp} \lVert f \rVert_{\pp}, \]
    where the supremum is taken over all $g \in L^{p'(\cdot)}(\Omega)$ with $\lVert g \rVert_{p'(\cdot)} \leq 1$, and 
    \begin{gather*}
         \frac{1}{k_{\pp}} = \lVert \chi_{\Omega_\infty} \rVert_\infty + \lVert \chi_{\Omega_1} \rVert_\infty + \lVert \chi_{\Omega_*} \rVert_\infty.
    \end{gather*}
\end{theorem}
\begin{remark}
    Note that $k_{\pp} = k_{p'(\cdot)}$ and $k_{\pp}^{-1} \leq 3$ for all $\pp \in \mathcal{P}(\R^n)$. 
\end{remark}

For the classical $L^p$ spaces, $L^{p'} \subset (L^p)^*$, the dual space of $L^p$, with equality (up to isomorphism) if and only if $p < \infty$. A similar statement is true for the variable Lebesgue spaces.
\begin{theorem}[\cite{VLS}, Theorem 2.80]
    Given $\Omega$ and $\pp \in \mathcal{P}(\Omega)$, $L^{p'(\cdot)}(\Omega) \cong (L^{\pp}(\Omega))^*$ if and only if $p_+ < \infty$. 
\end{theorem}

Finally, we state the definition of Log-H\"older continuity.  We will only use the $LH_0$ condition (in Section~\ref{Necessary but Not Sufficient}), but for completeness we include the definition of $LH_\infty$.

\begin{definition} \label{defn-LH}
Given $\pp \in \Pp(\rn)$, $p_+<\infty$, we say that $\pp$ is log-H\"older continuous if it satisfies the following two conditions:
\begin{enumerate}
    \item $\pp \in LH_0(\rn)$:  there exists a constant $C_0>0$ such that 
    \[ |p(x)-p(y)| \leq \frac{C_0}{-\log(|x-y)}, \qquad |x-y|<\frac{1}{2}; \]

    \item $\pp \in LH_\infty(\rn)$:  there exists a constants $C_\infty>0$ and $1\leq p_\infty<\infty$ such that for all $x\in \rn$,
    \[ |p(x)-p_\infty| \leq \frac{C_\infty}{\log(e+|x|)}. \]
\end{enumerate}
\end{definition}

\medskip

\subsection{Operators of Interest} \label{OperInt}
We first define some notation for norm inequalities for operators on the variable Lebesgue spaces.  Let $T$ be an operator on $\Lp$. We say that it satisfies a strong $(\pp,\qq)$ inequality, and write  $T : L^\pp \rightarrow L^\qq$, if 
\[ \|T f \|_{L^\qq(\rn)} \leq C \|f\|_{L^\pp(\rn)}.
\]
Denote the infimum of all such constants $C$ by $\|T\|_{\Lp\rightarrow L^\qq}$.
We say that $T$ satisfies a weak $(\pp,\qq)$ inequality, and write $T : L^\pp \rightarrow L^{\qq,\infty}$, if for every $\lambda>0$,
\[ \|\chi_{\{x\in \rn : |T f(x)|>\lambda \}} \|_{L^\qq(\rn)} \leq C \|f\|_{L^\pp(\rn)}.
\]
Denote the infimum of all such constants $C$ by $\|T\|_{\Lp\rightarrow L^{\qq,\infty}}$.

We now define the operators which we will work with. For complete information, see~\cite{MR3688149,MR3470665,GraBook,VLS}. We first define the fractional maximal operators. 
\begin{definition} \label{fractional maximal operator}
    Given $0 \leq \alpha < n$, we define the fractional maximal operator $M_\alpha$ by
    \[ M_\alpha f(x) = \sup_{Q \ni x} |Q|^{\frac{\alpha}{n}} \dashint_Q |f(y)| \, dy \hspace{8mm} \forall x \in \R^n, \]
    where the supremum is taken over cubes centered at $x$ with sides parallel to the coordinate axes. 
\end{definition}

\begin{remark}
In Definition~\ref{fractional maximal operator} we can also take the supremum over all cubes that contain $x$, and we will get an operator that is pointwise equivalent.  Similarly, if we define $M_\alpha$ using balls, either those centered at $x$ or containing $x$, we get an equivalent operator.
\end{remark}

We now define fractional  singular integral operators.
\begin{definition} \label{CZ operator}
Given $0 \leq \alpha < n$, we say that an operator $T_\alpha$, defined on measurable functions, is a fractional Calder\'on–Zygmund operator if $T_\alpha$ is bounded from $L^p(\R^n)$ to $L^q(\R^n)$ for some fixed $1 < p \leq q < \infty$, and for any $f \in L^p(\R^n)$ with compact support,
\begin{gather*}
    T_\alpha f(x) = \int_{\R^n} K_\alpha(x,y) f(y) \, dy, \hspace{5mm} x \not \in \supp{(f)}.
\end{gather*}
Here the kernel $K_\alpha(x,y)$ is  defined for all $(x,y) \in \R^n \times \R^n$, $x \neq y$, and satisfies the  size estimate
\begin{gather} \label{kernel singularity bound}
    |K_\alpha(x,y)| \leq \frac{C_0}{|x-y|^{n-\alpha}},
\end{gather}
and the smoothness estimate
\begin{gather} \label{kernel shift bound}
    |K_\alpha(x+h,y) - K_\alpha(x,y)| + |K_\alpha(x,y+h) - K_\alpha(x,y)| \leq C_0 \frac{|h|^\delta}{|x-y|^{n-\alpha+\delta}}
\end{gather}
for all $|h| < \frac{1}{2}|x-y|$ and some fixed $\delta > 0$. 
\end{definition}

Finally, we define the Riesz potentials, also referred to as fractional integral operators.

\begin{definition} \label{Riesz Potentials}
    Given $\alpha$, $0 < \alpha < n$, define the Riesz potential $I_\alpha$ to be the convolution operator
    \[ I_\alpha f(x) = \gamma(\alpha,n) \int_{\R^n} \frac{f(y)}{|x-y|^{n-\alpha}} \, dy, \]
    where
    \[ \gamma(\alpha,n) = \frac{\Gamma(\tfrac{n}2 - \tfrac\alpha2)}{\pi^{n/2}2^\alpha \Gamma(\tfrac\alpha2)}. \]
\end{definition}

\begin{remark} \label{remark:Ialpha-Talpha-bound}
    It is immediate that if $T_\alpha$ is a fractional CZO, then $|T_\alpha f(x)| \leq CI_\alpha (|f|)(x)$.  Thus sufficient conditions for the Riesz potential to be bounded also immediately imply that $T_\alpha$ is bounded.
\end{remark}

We will prove necessary conditions for fractional CZOs that satisfy a non-degeneracy condition.  This condition was introduced by Stein~\cite{SteinBook} for convolution-type singular integrals; see also~\cite{MR4319616}.

\begin{definition} \label{non-degenerate kernel}
    Given a fractional Calder\'on-Zygmund operator $T_\alpha$ with kernel $K_\alpha(x,y)$, we say that $T_\alpha$ has a non-degenerate kernel if there exists $a > 0$ and a unit vector $u_1$ such that for $x,y \in \R^n$, $x-y = tu_1$, $t \in \R$,
    \begin{gather} \label{non-degenerate condition}
        |K_\alpha(x,y)| \geq \frac{a}{|x-y|^{n-\alpha}}.
    \end{gather}
\end{definition}
The non-degeneracy condition \eqref{non-degenerate condition} is satisfied by many classic operators, such as the Hilbert transform, the Riesz transforms, and the Riesz potentials.

\begin{remark} \label{remark:extrapol}
We conclude this section by sketching why Log-H\"older continuity is not necessary for a variety of operators to be bounded on the variable Lebesgue spaces.  As we noted above, Lerner~\cite{lerner2005some} constructed an example of an exponent function $\pp$ that is discontinuous at $0$ and infinity, but such that the maximal operator is bounded on $\Lp$.  (See also~\cite{lerner2023boundedness}.)  By results of Diening (see~\cite[Theorem~4.37, Corollary~4.64]{VLS}), this implies that there exists $p_0>1$ such that $M$ is bounded on $L^{\pp/p_0}$ and on $L^{(\pp/p_0)'}$.  Therefore, by Rubio de Francia extrapolation in the scale of variable Lebesgue spaces, since CZOs satisfy weighted norm inequalities with respect to Muckenhoupt $A_p$ weights, they are also bounded on $\Lp$.  (See~\cite[Theorem~5.39]{VLS}, which is stated for convolution-type singular integrals, but which readily extends to more general CZOs.)  Similarly, given Lerner's example, it is possible to construct an exponent $\pp$ such that if $\qq$ is defined by~\eqref{conjugate exp}, then $\qq$ is not Log-H\"older continuous, but the maximal operator is bounded on $L^{(\qq/q_0)'}$ for some $q_0>1$.  Then again by extrapolation, the Riesz potentials $I_\alpha$ and fractional maximal operators are bounded from $\Lp$ to $L^\qq$ (see \cite[Theorem~5.46]{VLS} for Riesz potentials; bounds for $M_\alpha$ follow since it is dominated pointwise by $I_\alpha$).  By Remark~\ref{remark:Ialpha-Talpha-bound}, fractional CZOs satisfy the same bounds.
\end{remark}

\medskip
\section{Averaging Operators and the $K_0^\alpha$ Condition} \label{Averagin Operators and the K_0^alpha Condition}
\indent We now give some properties of the $K_0^\alpha(\R^n)$ condition.  We first recall the definition given in the Introduction.
\begin{definition} \label{K_0^alpha definition}
    Given $\Omega$, $0 \leq \alpha < n$, and $\pp \in \mathcal{P}(\Omega)$ such that $p_+<\frac{n}{\alpha}$, define $\qq$ by~\eqref{q dot defn}.  We say that $\pp \in K_0^\alpha(\Omega)$ if
    \[ \Kalconst{\pp}{\Omega} = \sup_{Q \subset \Omega} |Q|^{\frac{\alpha}{n}-1} \lVert \chi_Q \rVert_{p'(\cdot)} \lVert \chi_Q \rVert_{q(\cdot)}  < \infty, \]
    where the supremum is taken over cubes $Q \subset \Omega$ with sides parallel to the coordinate axes. In the case that $\alpha = 0$, we write $K_0(\Omega) = K_0^0(\Omega)$.
\end{definition}
The $K_0(\R^n)$ condition is due to Kopaliani \cite{Kopaliani2007}; the notation $K_0$ was introduced in \cite{VLS}. The $K_0$ and $K_0^\alpha$ conditions were implicit in Berezhon\u{\i}~\cite{Berezhnoi1999}. The explicit $K_0^\alpha$ condition was introduced in \cite{Chaffee}. As our next proposition shows, $\pp \in K_0^{\alpha}(\R^n)$ is equivalent to the fractional averaging operators over cubes being uniformly bounded from $L^{\pp}(\R^n) \to L^{\qq}(\R^n)$. Its proof can be found in \cite{VLS} for $\alpha = 0$ and in \cite{Berezhnoi1999} for $\alpha > 0$ in a more abstract setting. We omit the proof due to its similarity to the proof of Proposition \ref{K_0^alpha iff A_E^alpha bdd prop}. 
\begin{definition}
    Given a cube $Q$ and  $0 \leq \alpha < n$, define the fractional averaging operator $A_Q^\alpha$ by
    \[ A_Q^\alpha f(x) = |Q|^{\frac{\alpha}{n}} \dashint_Q f(y) \, dy \; \chi_Q(x). \]
\end{definition}
\begin{prop} \label{K_0^alpha and A_Q^alpha proposition}
    Given $0 \leq \alpha < n$ and $\pp \in \Pp(\Omega)$ such that $p_+<\frac{n}{\alpha}$, define $\qq$ by~\eqref{q dot defn}. Then $\pp \in K_0^\alpha(\Omega)$ if and only if the  operators $A_Q^\alpha$ satisfy $A_Q^\alpha : L^{\pp}(\Omega) \to L^{q(\cdot)}(\Omega)$ and are uniformly bounded for all cubes $Q \subset \Omega$. Moreover,  the operator norms satisfy 
    \begin{gather*}
        \sup_{Q} \lVert A_Q^\alpha \rVert_{L^\pp \rightarrow L^\qq} \leq K_{\pp} \Kalconst{\pp}{\Omega} \leq 4\Kalconst{\pp}{\Omega}.
    \end{gather*} 
\end{prop}
We need to extend the $K_0^\alpha$ condition to a general collection of sets. Hereafter, let $\E$ be a collection of sets $E \subset \R^n$ satisfying $0 < |E| < \infty$. 
\begin{definition} \label{Generalized K_0^alpha Definition}
   Given $0 \leq \alpha < n$ and $\pp \in \Pp(\Omega)$ such that $p_+<\frac{n}{\alpha}$, define $\qq$ by~\eqref{q dot defn}.  We say that $\pp \in K_0^\alpha(\mathcal{E})$ if
    \[ \Kalconst{\pp}{\E} = \sup_{E \in \mathcal{E}} |E|^{\frac{\alpha}{n}-1} \lVert \chi_E \rVert_{p'(\cdot)} \lVert \chi_E \rVert_{q(\cdot)} < \infty. \]
    In the case that $\alpha = 0$, we write $K_0(\E) = K_0^0(\E)$.
\end{definition}
Like the standard $K_0^\alpha(\R^n)$ condition, $\pp \in K_0^{\alpha}(\E)$ is equivalent to the uniform boundedness of the fractional averaging operators over each $E \in \E$, i.e.,
\[ A_E^\alpha f(x) = |E|^{\frac{\alpha}{n}} \dashint_E f(y) \, dy \; \chi_E(x). \]
\begin{prop} \label{K_0^alpha iff A_E^alpha bdd prop}
   Given $0 \leq \alpha < n$ and $\pp \in \Pp(\Omega)$ such that $p_+<\frac{n}{\alpha}$, define $\qq$ by~\eqref{q dot defn}. Given a  collection $\mathcal{E}$, $\pp \in K_0^\alpha(\mathcal{E})$ if and only if the  $A_E^\alpha$ satisfy $A_E^\alpha : L^{\pp}(\R^n) \to L^{q(\cdot)}(\R^n)$ and are uniformly bounded for all $E \in \mathcal{E}$. Moreover, the operator norms satisfy 
    \begin{gather} \label{operator norm of A_E^alpha}
        \sup_{E \in \E} \lVert A_E^\alpha \rVert_{L^\pp \rightarrow L^\qq} \leq K_{\pp} \Kalconst{\pp}{\E} \leq 4\Kalconst{\pp}{\E}.
    \end{gather}  
\end{prop}
\begin{proof}
    The proof is essentially the same as that for cubes; for completeness we include the short proof. If $\pp \in K_0^\alpha(\mathcal{E})$, then for all $E \in \mathcal{E}$,
    \begin{multline*}
        \lVert A_E^\alpha f \rVert_{q(\cdot)} = |E|^{\frac{\alpha}{n}-1} \int_{\R^n} f(x) \chi_E (x) \, dx \lVert \chi_E \rVert_{q(\cdot)} \\
        \leq K_{\pp} |E|^{\frac{\alpha}{n}-1} \lVert f \rVert_{\pp} \lVert \chi_E \rVert_{p'(\cdot)} \lVert \chi_E \rVert_{q(\cdot)} 
        \leq K_{\pp} \Kalconst{\pp}{\E} \lVert f \rVert_{\pp}.
    \end{multline*}
    Since this inequality holds for all $f \in L^{\pp}(\R^n)$ and $E \in \E$, $K_{\pp} \Kalconst{\pp}{\E}$ is a uniform bound for the operator norms of $\{ A_E^{\alpha} \}_{E \in \E}$. Conversely, suppose that the fractional averaging operators are uniformly bounded with constant $B > 0$. Then by duality there exists $g \in L^{\pp}$, $\lVert g \rVert_{\pp} \leq 1$, such that
    \begin{multline*}
        |E|^{\frac{\alpha}{n}-1} \lVert \chi_E \rVert_{q(\cdot)} \lVert \chi_E \rVert_{p'(\cdot)} 
        \leq k_{\pp}^{-1} |E|^{\frac{\alpha}{n}-1} \lVert \chi_E \rVert_{q(\cdot)} \int_{\R^n} \chi_E g \, dx  \\
        = k_{\pp}^{-1} \lVert A_E^\alpha g \rVert_{q(\cdot)} 
        \leq k_{\pp}^{-1} B \lVert g \rVert_{\pp} 
        \leq k_{\pp}^{-1} B \leq 3B.
    \end{multline*}
\end{proof}
We now give some of the consequences of  the $K_0^{\alpha}$ condition. In the classical Lebesgue spaces $(1 \leq p < \infty)$, for any set $E \subset \R^n$ we always have
\[ \lVert \chi_E \rVert_{L^p(\R^n)} = |E|^{\frac{1}{p}}. \]
Finding a similar expression is not possible in general in the variable Lebesgue spaces. However, for $\pp \in K_0^{\alpha}(\E)$, we can obtain an expression for $\lVert \chi_E \rVert_{L^{\pp}(\R^n)}$ up to uniform constants via the harmonic mean. This was shown in \cite[Theorem~4.5.7]{Diening} for cubes.  
\begin{definition} \label{Harmonic Mean Definition}
    Given $\pp \in \mathcal{P}(\R^n)$ and a set $E$, $0 < |E| < \infty$, define the harmonic mean $p_E$ of $\pp$ on $E$ by
    \begin{gather*}
        \frac{1}{p_E} = \dashint_{E} \frac{1}{p(x)} \, dx.
    \end{gather*}
\end{definition}
\begin{remark}
    For the H\"older conjugate $p'(\cdot)$, there is no ambiguity in writing $p'_E$ to denote the harmonic mean of $p'(\cdot)$ on $E$ since 
    \[ \frac{1}{p'_E} = \frac{1}{(p_E)'} = 1 - \frac{1}{p_E}. \]
\end{remark}
\begin{prop} \label{norm and harmonic mean proposition}
   If $\pp \in K_0(\mathcal{E})$, then for all $E \in \E$, 
    \begin{gather} \label{norm and harmonic mean equivalence eqn}
        \frac{1}{2K_{\pp}} |E|^{1/p_E} \leq \lVert \chi_E \rVert_{\pp} \leq \frac{2K_{\pp}^2 \Kconst{\pp}{\E}}{k_{\pp}} |E|^{1/p_E}.
    \end{gather}
\end{prop}
\begin{proof}
    The first inequality in~\eqref{norm and harmonic mean equivalence eqn} follows by modifying the proof of \cite[Lemma~3.1]{GreedyBases}, which was derived from the proof given in \cite{Diening}. Fix a set $E \in \E$. We have that 
    \begin{align*}
        \frac{1}{p_E} + \frac{1}{p'_E} =  \dashint_E \frac{1}{p(x)} + \frac{1}{p'(x)} \, dx 
        = 1.
    \end{align*}
    By Jensen's inequality,
    \begin{multline*}
        \left(\frac{1}{|E|}\right)^{\frac{1}{p'_E}} = \exp \left(\log \left[\left(\frac{1}{|E|}\right)^{\dashint_E \frac{1}{p'(x)} \, dx} \right]\right) = \exp \left(\left[\dashint_E \frac{1}{p'(x)} \, dx\right] \log \left(\frac{1}{|E|}\right)\right) \\
        = \exp \left(\dashint_E \log \left[\left(\frac{1} {|E|}\right)^{\frac{1}{p'(x)}}\right] \, dx\right) 
        \leq \dashint_E \left(\frac{1}{|E|}\right)^{\frac{1}{p'(x)}} \,dx.
    \end{multline*}
    Now by H\"older's inequality (Theorem \ref{Holders}),
    \begin{multline} \label{|E| 1/P_E Relation}
       |E|^{\frac{1}{p_E}} = |E| \left(\frac{1}{|E|} \right)^{\frac{1}{p'_E}} \\
       \leq |E| \, \dashint_E \left(\frac{1}{|E|}\right)^{\frac{1}{p'(x)}} \, dx 
       \leq K_{\pp} \, \lVert \chi_E \rVert_{\pp} \lVert |E|^{-1/p'(\cdot)} \chi_E \rVert_{p'(\cdot)}. 
    \end{multline}
    Define $E'_\infty = \{ x \in E : p'(x) = \infty\}$. Then,
    \begin{multline*}
       \rho_{p'(\cdot)} \left(\frac{|E|^{-\frac{1}{p'(x)}} \chi_E}{2} \right) = \int_{E \setminus E'_\infty} \left(\frac{|E|^{-\frac{1}{p'(x)}}}{2}\right)^{p'(x)} \, dx + \frac{1}{2} \lVert |E|^{-\frac{1}{p'(\cdot)}} \rVert_{L^\infty(E'_\infty)}
       \\ 
       \leq \frac{1}{2} \left(\int_{E \setminus E'_\infty} |E|^{-1} \, dx + 1 \right) \leq 1.
    \end{multline*}
    This implies that $\lVert |E|^{-1/p'(\cdot)} \chi_E \rVert_{p'(\cdot)} \leq 2$, which by \eqref{|E| 1/P_E Relation} gives the first inequality in \eqref{norm and harmonic mean equivalence eqn}. 
    
    \indent For the second inequality we first note that $\pp \in K_0(\mathcal{E})$ if and only if $p'(\cdot) \in K_0(\mathcal{E})$ by the symmetry of the definition. Thus, assuming $\pp \in K_0(\E)$ is equivalent to the averaging operators $A_E$ being uniformly bounded on $L^{p'(\cdot)}$, and $\lVert A_E \rVert_{L^{p'(\cdot)} \to L^{p'(\cdot)}} \leq K_{\pp} \Kconst{\pp}{\E}$. By duality, there exists $g \in L^{p'(\cdot)}$, $\lVert g \rVert_{p'(\cdot)} = 1$, such that
    \begin{multline*}
        \frac{\lVert \chi_E \rVert_{\pp}}{|E|^{1/p_E}} \leq \frac{|E|^{-\frac{1}{p_E}}}{k_{\pp}} \int_E g \, dx = \frac{|E|^{\frac{1}{p'_E}}}{k_{\pp}} \dashint_E g \, dx.
        \leq \frac{2K_{\pp}}{k_{\pp}} \lVert \chi_E \rVert_{p'(\cdot)} \dashint_E g \, dx \\
        = \frac{2K_{\pp}}{k_{\pp}} \lVert A_E g \rVert_{p'(\cdot)} 
        \leq \frac{2K_{\pp}^2 \Kconst{\pp}{\E}}{k_{\pp}} \lVert g \rVert_{p'(\cdot)} = \frac{2K_{\pp}^2 \Kconst{\pp}{\E}}{k_{\pp}}.
    \end{multline*}
    The second inequality follows from the first inequality in \eqref{norm and harmonic mean equivalence eqn} applied to the exponent $p'(\cdot)$. This implies the second inequality in \eqref{norm and harmonic mean equivalence eqn}, and we are done.
\end{proof}
\begin{lemma} \label{K_0^alpha implies K_0 lemma}
  Given $0 \leq \alpha < n$ and $\pp\in \Pp(\Omega)$, $p_+<\frac{n}{\alpha}$, define $\qq$ by~\eqref{q dot defn}.  Then  $\pp \in K_0^\alpha(\E)$ if and only if $\pp,q(\cdot) \in K_0(\E)$.
\end{lemma}
\begin{proof}
    This is immediate for $\alpha = 0$, so we assume $0 < \alpha < n$. Suppose $\pp \in K_0^\alpha(\E)$. Since
    \[  \frac{1}{p(x)} = \frac{\alpha}{n}+\frac{1}{q(x)}, \qquad x \in \R^n,
\]
for each $E \in \E$, by the generalized H\"older's inequality (Corollary \ref{gen Holder}) and the definition of $K_0^\alpha(\E)$, we get
    \[ |E|^{-1} \lVert \chi_E \rVert_{\pp} \lVert \chi_E \rVert_{p'(\cdot)} \leq K |E|^{-1} \lVert \chi_E \rVert_{q(\cdot)} \lVert \chi_E \rVert_{\frac{n}{\alpha}} \lVert \chi_E \rVert_{p'(\cdot)} \leq K \Kalconst{\pp}{\E} < \infty. \]
    Since this bound is independent of the choice of $E$, $\pp \in K_0(\E)$. To show $q(\cdot) \in K_0(\E)$, note that since
    \[ \frac{1}{q'(x)} - \frac{1}{p'(x)} = \frac{\alpha}{n}, \hspace{10mm} x \in \R^n, \]
  again by the generalized H\"older's inequality, for all $E \in \E$ we have
    \[ |E|^{-1} \lVert \chi_E \rVert_{\qq} \lVert \chi_E \rVert_{q'(\cdot)} \leq K |E|^{-1} \lVert \chi_E \rVert_{\qq} \lVert \chi_E \rVert_{\frac{n}{\alpha}} \lVert \chi_E \rVert_{p'(\cdot)} \leq K \Kalconst{\pp}{\E} < \infty. \]
    Thus $\pp,q(\cdot) \in K_0(\E)$. 

    The converse follows from Proposition \ref{norm and harmonic mean proposition}. If $\pp, \qq \in K_0(\E)$ then for all $E \in \E$,
    \[ \lVert \chi_E \rVert_{p'(\cdot)} \lVert \chi_E \rVert_{\qq} \lesssim |E|^{\frac{1}{p'_E}} |E|^{\frac{1}{q_E}} = |E|^{1- \frac{\alpha}{n}}. \]
  The last inequality follows since
    \[ \frac{1}{p'_E} + \frac{1}{q_E} = \dashint_E \frac{1}{p'(x)} + \frac{1}{q(x)} \, dx = 1 - \frac{\alpha}{n}. \] 
    Since this bound is independent of the choice of $E$, $\pp \in K_0^\alpha(\R^n)$.
\end{proof}
\begin{corollary} \label{K_0^alpha implies norm equivalence}
    If $\pp \in K_0^\alpha (\E)$ then for all $E \in \E$,
    \begin{gather*}
        \lVert \chi_E \rVert_{\pp} \approx |E|^{\frac{1}{p_E}} \hspace{4mm} , \hspace{4mm} \lVert \chi_E \rVert_{p'(\cdot)} \approx |E|^{\frac{1}{p'_E}}, \\
        \lVert \chi_E \rVert_{q(\cdot)} \approx |E|^{\frac{1}{q_E}} \hspace{4mm} \text{ and } \hspace{4mm} \lVert \chi_E \rVert_{q'(\cdot)} \approx |E|^{\frac{1}{q'_E}}. 
    \end{gather*} 
\end{corollary}
\begin{proof}
     By the symmetry of the definition, $\pp \in K_0(\E)$ if and only if $p'(\cdot) \in K_0(\E)$. Thus, this is a direct consequence of Lemma \ref{K_0^alpha implies K_0 lemma} and Proposition \ref{norm and harmonic mean proposition}.
\end{proof}
We now prove a preliminary result on finding sets with minimal harmonic mean. Along with Corollary \ref{K_0^alpha implies norm equivalence}, this will be an essential tool in the proof of Theorems \ref{Maximal Bounded Exponent Theorem} and \ref{SIO Bounded Exponent Theorem}.
\begin{lemma} \label{minimal harmonic mean lemma}
    Let $\pp \in \mathcal{P}(\R^n)$ and $D \subset \R^n$ be a cube of radius of $R > 0$. Then for each $0 < r \leq R$ and $E \subset \R^n$ satisfying $|D \setminus E| < (2r)^n$ and $p_+(D \cap E) < \infty$, there exists a cube $Q^* \subset D$ of radius $r$ such that for any $m \in \N$ with $mr \leq R$, $p_{Q^* \cap E} \leq p_{Q \cap E}$ for all cubes $Q \subset D$ of radius $mr$.  
\end{lemma}
\begin{proof}
    We first prove the existence of such a cube for $m = 1$. Given $\pp \in \mathcal{P}(\R^n)$ and a cube $D = D(x,R)$, let $0 < r \leq R$ and $E \subset \R^n$ satisfy the above hypotheses. Define the cube $F = F(x,R-r) \subset D$ and the function $f : F \to [1,\infty)$ by 
    \begin{gather*}
        f(y) = \left( \frac{1}{|Q(y,r) \cap E|} \int_{Q(y,r) \cap E} \frac{1}{p(x)} \, dx\right)^{-1} = p_{Q(y,r) \cap E}.
    \end{gather*}
    Since $|D \cap E^c| < 2^n r^n$ implies $|Q(y,r) \cap E| > 0$,  and since $p_+(D \cap E) < \infty$, the function $f$ is well defined. By a standard property of the integral we have that $f$ is continuous. Because the domain is compact, there exists a minimum $a \in F$. Thus, there exists a cube $Q^*(a,r) \subset D$ such that for any other cube $Q(b,r) \subset D$ we have $p_{Q^*(a,r) \cap E} = f(a) \leq f(b) = p_{Q(b,r) \cap E}$. This proves the case when $m = 1$.\\
    \indent Now let $m \in \N$ satisfy $mr \leq R$ and choose any cube $K = K(c,mr) \subset D$. Denote $Q^* = Q^*(a,r)$. Form the collection of $m^n$ dyadic subcubes of $K$, $K_1, \dots , K_{m^n}$, with side length $r$. Then
    \begin{multline*}
        \frac{1}{p_{K \cap E}} = \frac{1}{|K \cap E|} \int_{K \cap E} \frac{1}{p(x)} \, dx 
        = \frac{1}{|K \cap E|} \sum_{i=1}^{m^n} \int_{K_i \cap E} \frac{1}{p(x)} \, dx \\
        = \sum_{i=1}^{m^n} \frac{|K_i \cap E|}{|K \cap E|} \frac{1}{|K_i \cap E|} \int_{K_i \cap E} \frac{1}{p(x)} \, dx 
        = \sum_{i=1}^{m^n} \frac{|K_i \cap E|}{|K \cap E|} \frac{1}{p_{K_i \cap E}} 
        \leq \frac{1}{p_{Q^* \cap E}}.
    \end{multline*}
   Thus, $p_{Q^* \cap E} \leq p_{K \cap E}$.
\end{proof}
We remark that Lemma \ref{minimal harmonic mean lemma} is true only for integer scalings. 
\begin{example}
    Let $Q = Q(0,1) \subset \R^2$, $E = \R^2$, and fix $\frac{1}{2} < r < 1$. Let $Q' = Q(0,2r-1)$. Define $p(x) = 2 \chi_{Q'}(x) + \chi_{Q \setminus Q'}(x)$. Notice that for any cube $Q(x,r) \subset Q$ we have $Q' \subset Q(x,r)$. It follows that the harmonic mean of every such cube is the same, and hence the minimum is precisely 
    \begin{gather*}
        p_{Q(x,r)} = \left (\frac{1}{4r^2}\int_{Q(x,r)} \frac{1}{p(x)} \, dx \right)^{-1} = \frac{2r^2}{4r - 2r^2 - 1}.
    \end{gather*}
    However,
    \begin{gather*}
        p_Q = \left(\frac{1}{4} \int_Q \frac{1}{p(x)} \, dx \right)^{-1} = \frac{2}{4r-4r^2+1}.
    \end{gather*}
    Direct computations show $p_Q \leq p_{Q(x,r)}$ which contradicts the conclusion of Lemma \ref{minimal harmonic mean lemma}.
\end{example}
\section{Fractional Maximal Operators} \label{Fractional Maximal Operators}
In this section we prove Theorems \ref{Maximal K_0^alpha theorem} and \ref{Maximal Bounded Exponent Theorem}.

\begin{proof}[Proof of Theorem \ref{Maximal K_0^alpha theorem}]
 To show that $\pp \in K_0^\alpha(\R^n)$ whenever $M_\alpha : L^{\pp}(\R^n) \to L^{\qq,\infty}(\R^n)$, we will prove the fractional averaging operators are uniformly bounded since, by Proposition \ref{K_0^alpha and A_Q^alpha proposition}, this is equivalent to $\pp \in K_0^\alpha(\R^n)$.   Fix a cube $Q \subset \R^n$. By definition we have
    \[ M_\alpha f(x) \geq |Q|^{\frac{\alpha}{n}} \dashint_Q |f(y)| \, dy \]
    for all $x \in Q$. Thus, for $0 < \lambda < |Q|^{\frac{\alpha}{n}} \dashint_Q f \, dy$, our assumption that $M_\alpha : L^{\pp} \to L^{q(\cdot),\infty}$ implies
    \[ \lVert \chi_Q \rVert_{q(\cdot)} \leq \lVert \chi_{\{M_\alpha f  > \lambda\}} \rVert_{q(\cdot)} \leq \frac{\lVert M_\alpha \rVert_{L^{\pp} \to L^{\qq,\infty}}}{\lambda} \lVert f \rVert_{\pp}. \]
  If we take the supremum over all such $\lambda$, we get
    \begin{gather} \label{off-diagonal K_0^alpha ineq}
        \lVert A_Q^\alpha \rVert_{q(\cdot)} 
        = |Q|^{\frac{\alpha}{n}} \left(\dashint_Q f \, dy \right) \lVert \chi_Q \rVert_{q(\cdot)} 
        \leq \lVert M_\alpha \rVert_{L^{\pp} \to L^{\qq,\infty}} \lVert f \rVert_{\pp}.
    \end{gather}  
    Since the constant $\lVert M_\alpha \rVert_{L^{\pp} \to L^{\qq,\infty}}$ on the righthand side does not depend on our choice of $Q$, the fractional averaging operators $A_Q^\alpha$ are uniformly bounded from $L^{\pp}(\R^n) \to L^{q(\cdot)}(\R^n)$, and so, by Proposition \ref{K_0^alpha and A_Q^alpha proposition},   $\pp \in K_0^\alpha (\R^n)$. 
\end{proof}
\begin{remark}
    By our proof of Theorem \ref{Maximal K_0^alpha theorem}, if $M_\alpha : L^{\pp} \to L^{\qq,\infty}$, then
    \[ \sup_{Q} \lVert A_Q^\alpha \rVert_{L^{\pp} \to L^{\qq}} \leq \lVert M_\alpha \rVert_{L^{\pp} \to L^{\qq,\infty}}. \]
    Thus, Proposition \ref{K_0^alpha and A_Q^alpha proposition} yields
    \[ \Kalconst{\pp}{\R^n} \leq 2 \lVert M_\alpha \rVert_{L^{\pp} \to L^{\qq,\infty}}. \]
\end{remark}
We now prove an analog to Theorem \ref{Maximal K_0^alpha theorem} with the generalized $K_0^\alpha(\E)$ condition (Definition \ref{Generalized K_0^alpha Definition}). It requires an additional assumption on $\mathcal{E}$, which we call the cube property.
\begin{definition} \label{cube property}
    We say $\mathcal{E}$ has the cube property if for each $E \in \mathcal{E}$ there exists a cube $Q$ such that $E \subset Q$ and $|E| \geq \frac{1}{2} |Q|$.
\end{definition}
\begin{prop} \label{Generalized K_0^alpha Theorem}
    Given  $0 \leq \alpha \leq n$ and $\pp \in \mathcal{P}(\R^n)$ such that $p_+<\frac{n}{\alpha}$, define $q(\cdot)$ by~\eqref{q dot defn}.  If $\mathcal{E}$ has the cube property and   $M_\alpha : L^{\pp}(\R^n) \to L^{q(\cdot),\infty}(\R^n)$, then $\pp \in K_0^\alpha(\mathcal{E})$.  
\end{prop}
\begin{proof}
    By Theorem \ref{Maximal K_0^alpha theorem}  we have that $\pp \in K_0^\alpha(\R^n)$. Fix $E \in \mathcal{E}$ and a cube $Q$ such that $E \subset Q$ and $|E| \geq \frac{1}{2}|Q|$. Then
    \begin{align*}
        |E|^{\frac{\alpha}{n}-1} \lVert \chi_E \rVert_{p'(\cdot)} \lVert \chi_E \rVert_{q(\cdot)} \leq 2^{1 - \frac{\alpha}{n}} |Q|^{\frac{\alpha}{n}-1} \lVert \chi_Q \rVert_{p'(\cdot)} \lVert \chi_Q \rVert_{q(\cdot)} \leq 2^{1-\frac{\alpha}{n}} \Kalconst{\pp}{\R^n}.
    \end{align*}
    Since this bound is uniform for all $E \in \mathcal{E}$, we have that $\pp \in K_0^\alpha(\mathcal{E})$. 
\end{proof}
\begin{corollary} \label{fractional strong type gives norm equivalence corollary}
    Given  $0 \leq \alpha \leq n$ and $\pp \in \mathcal{P}(\R^n)$ such that $p_+<\frac{n}{\alpha}$, define $q(\cdot)$ by~\eqref{q dot defn}.
    If $\mathcal{E}$ has the cube property and  $M_\alpha : L^{\pp}(\R^n) \to L^{q(\cdot)}(\R^n)$, then for all $E \in \mathcal{E}$,
    \begin{gather*}
        \lVert \chi_E \rVert_{\pp} \approx |E|^{\frac{1}{p_E}}, \hspace{4mm} \lVert \chi_E \rVert_{p'(\cdot)} \approx |E|^{\frac{1}{p'_E}}, \\
        \lVert \chi_E \rVert_{q(\cdot)} \approx |E|^{\frac{1}{q_E}}, \hspace{4mm} \text{ and } \hspace{4mm} \lVert \chi_E \rVert_{q'(\cdot)} \approx |E|^{\frac{1}{q'_E}}.
    \end{gather*} 
\end{corollary}
\begin{proof}
   By Chebyshev's inequality,  strong-type bounds imply weak-type bounds.  Thus, by Theorem \ref{Maximal K_0^alpha theorem} and Proposition~\ref{Generalized K_0^alpha Theorem} we have that  $\pp \in K_0^\alpha(\E)$. Hence, this result follows from Corollary \ref{K_0^alpha implies norm equivalence}.
\end{proof}
Our motivation for the proof of Theorem \ref{Maximal Bounded Exponent Theorem} is the following example, which shows that for any $0 \leq \alpha < n$, $M_\alpha : L^1(\R^n) \not \to L^{\frac{n}{n-\alpha}}(\R^n)$.
\begin{example} \label{M not bdd on L1 example}
    \normalfont Fix $\alpha$, $0 \leq \alpha < n$, let $Q = Q(0,1/2)$, and define $f = \chi_{Q}$. It is easy to see that $f \in L^1(\R^n)$ and $\lVert f \rVert_1 = 1$. For $x \not \in Q$, notice that $Q \subset Q_x = Q\left(0, |x| \right)$. Therefore,
    \[ M_\alpha f(x) \geq |Q_x|^{\frac{\alpha}{n}-1} \int_{Q_x} |f(x)| \, dx = (2^n |x|^n)^{\frac{\alpha}{n}-1} = 2^{\alpha - n} |x|^{\alpha - n}. \]
    Hence by a change to polar coordinates, 
    \begin{multline} \label{eqn:key-est}
        \lVert M_\alpha f \rVert_{\frac{n}{n-\alpha}}^{\frac{n}{n-\alpha}} = \int_{\R^n} M_\alpha f(x)^{\frac{n}{n-\alpha}} \, dx \\
        \geq \int_{\R^n \setminus Q} (2^{\alpha - n} |x|^{\alpha - n})^{\frac{n}{n-\alpha}} \, dx  \geq \frac{C(n)}{2^{n}} \int_{\sqrt{n}}^\infty r^{-1} \, dr 
        = \infty.
    \end{multline} 
Thus $M_\alpha f \not \in L^{\frac{n}{n-\alpha}}(\R^n)$ and so $M : L^1(\R^n) \not \to L^{\frac{n}{n-\alpha}}(\R^n)$. 
\end{example}

The key fact used to construact Example \ref{M not bdd on L1 example} is that $|x|^{\alpha - n} \chi_{\R^n \setminus Q} \not \in L^{\frac{n}{n-\alpha}}(\R^n)$. However, we do not have in general that $|x|^{\alpha - n} \chi_{\R^n \setminus Q} \not \in L^{\qq}(\R^n)$ when $p_- = 1$, i.e. when $q_- = \frac{n}{n-\alpha}$. Indeed, given $\alpha$, $0 \leq \alpha < n$, $p(x) = \chi_{Q} + \frac{n+1}{\alpha+1} \chi_{\R^n \setminus Q}$ is an exponent function satisfying $p_- = 1$ and $|x|^{\alpha - n} \in L^{\qq}(\R^n)$. 

To work around this, we use the fact that  $p_- = 1$ and apply the Lebesgue differentiation theorem to construct a sequence of cubes $D_k$ where the exponent function approaches 1 on a large proportion of the cube. This proportion increases with the index $k$. On each cube $D_k$ we define a function $f_k$ to be the characteristic function of a smaller cube intersected with the subset of $D_k$ where the exponent is near 1. We choose this cube so that it has minimal harmonic mean with respect to cubes of the same size contained in $D_k$. By Corollary \ref{fractional strong type gives norm equivalence corollary} this is equivalent to a cube with minimal norm up to uniform constants. Starting from the support of $f_k$ we next construct a chain of disjoint cubes that form what we call $(t,u)$ pairs and bound the integral of $M_\alpha f_k$ on these cubes by means of  a pointwise estimate of $M_\alpha f_k$ on each cube in the chain.   Using these bounds,  we show that by increasing the length of the chain of disjoint cubes, we get an estimate analogous to~\eqref{eqn:key-est} and so get that the modular of $M_\alpha f_k$ can be made arbitrarily large as $k\rightarrow \infty$.  This shows that $M_\alpha$ does not satisfy strong 
$(\pp,\qq)$ bounds.

\begin{remark}
    We want to emphasize that the proof we give for the fractional maximal operator below is somewhat more complicated than is aboslutely necessary.  In particular, we do not need to use Lemma~\ref{Ptwise Average Estimate Lemma}, and in our construction of the cubes in the proof we could use the standard orthonormal basis rather than constructing a basis around a specific vector $u_1$.  We have chosen this approach since it essentially allows us to unite the proofs for the fractional maximal operator and for fractional CZOs and thus avoid doing the construction twice, first in the simpler case and then in full generality.
\end{remark}

\medskip

\begin{definition} \label{QP Pair Definition}
    Given $t \in \R$ and a vector $u \in \R^n$, we say the cubes $Q = Q(x,r)$ and $P = Q(y,r)$ form an $(t,u)$ pair if $y = x + tr\sqrt{n}u$. 
\end{definition}

The next lemma gives us a decay estimate for the fractional maximal operator that we can apply to the chain of cubes we will construct.

\begin{lemma} \label{Ptwise Average Estimate Lemma}
    Given $0 \leq \alpha < n$, $t \geq 4$, and a unit vector $u$, let $Q$ and $P$ be cubes of radius $r > 0$ that form a $(t,u)$ pair. Then for a measurable function $f$ and for all $y \in P$, 
    \begin{gather}
        M_\alpha f(y) \geq \left(\frac{(t+2)\sqrt{n}}{2}\right)^{\alpha-n} |Q|^{\frac{\alpha}{n}} \dashint_{Q} |f(x)| \, dx.
    \end{gather}
\end{lemma}
\begin{proof}
    Fix $t \geq 4$, a unit vector $u$, and cubes $Q = Q(x_0,r)$ and $P = Q(y_0,r)$ that form a $(t,u)$ pair. By the triangle inequality, we have that 
    \[ Q,\,P \subset Q_B = Q\left(\frac{x_0 + y_0}{2}, \frac{(t+2)r\sqrt{n}}{2}\right) \]
    and
    \[ |Q| = |P| = \left(\frac{2}{(t+2)\sqrt{n}}\right)^n |Q_B|. \] 
    Thus, for any measurable function $f$ and $y \in P$,
    \[ M_\alpha f(y) \geq |Q_B|^{\frac{\alpha}{n} - 1} \int_{Q_B} f(x) \, dx \geq \left(\frac{(t+2)\sqrt{n}}{2}\right)^{\alpha-n} |Q|^{\frac{\alpha}{n}} \dashint_Q |f(x)| \, dx. \]
\end{proof}
We now define some useful notation and proceed with the proof of Theorem \ref{Maximal Bounded Exponent Theorem}.   
\begin{definition}
    Given an orthonormal basis $u_1, \dots , u_n$ of $\R^n$, let $Q = Q(x,r) \subset \R^n$ be a cube with each of its $n-1$ dimensional boundary hyperplanes normal to some $u_i$, $1 \leq i \leq n$. We define the lower corner of $Q$ to be the point
    \begin{gather*}
        lc \, (Q) = x - \sum_{i=1}^n r u_i.
    \end{gather*}
\end{definition}
\begin{proof}[Proof of Theorem \ref{Maximal Bounded Exponent Theorem}]
    We prove this by contradiction.  Suppose  to the contrary that $p_- = 1$ but  $M_\alpha : L^{\pp}(\R^n) \to L^{\qq}(\R^n)$. Then there exists a constant $C > 1$ such that for all $f \in L^{\pp}$, $\lVert M_\alpha f \rVert_{\pp} \leq C \lVert f \rVert_{\qq}$. We will show this leads to a contradiction. Fix a unit vector $u_1$. (Note that we could choose $e_1$, but we make our construction more general than necessary so that it may be applied to fractional CZOs in the next section.)
    
    Fix $t > 4$ and for each $k \in \N$ let
\begin{gather} \label{beta_k}
    \beta_k = \frac{n^2k + n}{n^2k+\alpha} > 1.
\end{gather}
Define $E_k = \{x \in \R^n : p(x) < \beta_k\}$. Since $p_- = 1$, this set has positive measure and hence, by the Lebesgue differentiation theorem applied to $\chi_{E_k}$, there exists $x_k \in E_k$ such that
\begin{gather*}
    \lim_{r \to 0^+} \frac{|B(x_k,r) \cap E_k|}{|B(x_k,r)|} = 1.    
\end{gather*}
Thus, for each $k$ there exists $S_k$, $0 < S_k < \frac{1}{2}$, such that
\begin{gather*} 
    \frac{|B(x_k,S_k) \cap E_k|}{|B(x_k,S_k)|} > 1 - \frac{2^{-n(k-1)}} {v(n) (tn+2\sqrt{n})^{n}},
\end{gather*}
where $v(n)$ is the measure of an $n$-dimensional unit ball. Let $R_k = n^{-1/2} S_k$ and let $D_k = Q_{u_1}(x_k,R_k)$. Here $Q_{u_1}(x_k,R_k)$ denotes a closed cube centered at $x_k$ with radius $R_k$ which has $u_1$ as a normal vector to one of its $n-1$ dimensional boundary hyperplanes. Then $D_k \subset B(x_k,S_k)$ and 
\[ \frac{|D_k \setminus E_k|}{|D_k|} \leq \frac{|B(x_k,S_k) \setminus E_k|}{\frac{2^n}{v(n) \sqrt{n}^n} |B(x_k,S_k)|} < \frac{2^{-n(k-1)}} {v(n) (tn+2\sqrt{n})^{n}} \cdot \frac{v(n) \sqrt{n}^n}{2^n} =  \frac{2^{-nk}}{(t\sqrt{n} + 2)^n}. \] 
This implies 
\begin{gather} \label{percent Q_k in E_k}
    \frac{|D_k \cap E_k|}{|D_k|} > 1 - 2^{-nk}{(t\sqrt{n}+2)}^{-n}.
\end{gather}
Now let 
\begin{gather*}
    r_k = \frac{R_k}{2^{k-1}(t\sqrt{n}+2)}.
\end{gather*}
Define $Q_1^{k} = Q_{u_1}(y_k, r_k) \subset Q_k$ to be the cube such that $Q^k_1 \cap E_k$ has the minimal harmonic mean of all cubes contained in $D_k$ with radius $r_k$. The existence of such a cube is guaranteed by Lemma \ref{minimal harmonic mean lemma}. Let $u_2, \dots , u_n$ be a collection of vectors such that $u_1, \dots , u_n$ forms an orthonormal basis of $\R^n$ with each $u_i$ normal to a face of $D_k$ and $y_k + R_k u_i \in D_k$ for each $i$, $1 \leq i \leq n$. This may involve replacing $u_1$ by $-u_1$, but this has no effect on our arguments. The following image depicts the choice of $u_1$ and $u_2$ in $\R^2$. \\ 
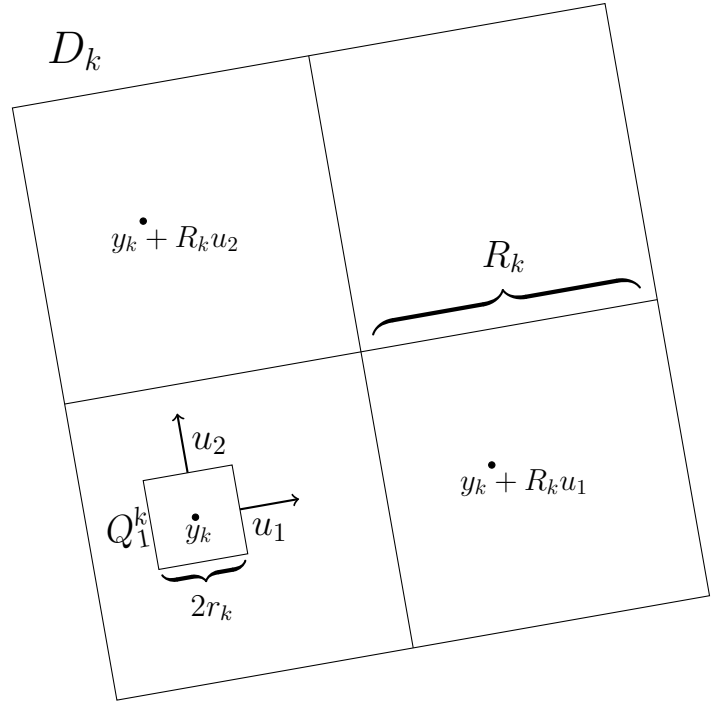
\begin{center}
\begin{tikzpicture}[scale=0.8]
        \draw[rotate around = {10:(5,5)}] (0,0) rectangle (10,10);
            \node[] at (0.25,10) {\Large $D_k$};
        \draw[rotate around = {10:(5,5)}] (0,5) -- (10,5);
        \draw[rotate around = {10:(5,5)}] (5,0) -- (5,10);
        \draw[rotate around = {10:(2.25,2.25)}] (1.5,1.5) rectangle (3,3);
            \node[rotate around = {10:(2.25,2.25)}] at (0.6,2.5) {\large $Q^k_1$};
        \draw (2.25,2.25) circle[radius=1.5pt]
            node[pos=0.5, above=56pt, right=55pt] {\large $y_k$};
        \fill (2.25,2.25) circle[radius=1.5pt];
        \draw[thick,->,rotate around = {10:(2.25,2.25)}] (3,2.25) -- (4,2.25)
            node[pos=0.5, below=1pt] {\large $u_1$};
        \draw[thick,->,rotate around = {10:(2.25,2.25)}] (2.25,3) -- (2.25,4)
            node[pos=0.5, right=-1pt] {\large $u_2$};
        \draw (7.17404,3.11824) circle[radius=1.5pt]
            node[pos=0.5, above=80.5pt, right=185pt] {\large $y_k + R_k u_1$};
        \fill (7.17404,3.11824) circle[radius=1.5pt];
        \draw (1.38176,7.17404) circle[radius=1.5pt]
            node[pos=0.5, above=195.5pt, right=19.9pt] {\large $y_k + R_k u_2$};
        \fill (1.38176,7.17404) circle[radius=1.5pt];
        \draw[ultra thick, decorate, decoration = {calligraphic brace, amplitude=8pt},rotate around = {10:(5,5)}] (5.25,5.25) -- (9.75,5.25)
            node[pos=0.49, above=10pt] {\large $R_k$};
        \draw[ultra thick, decorate, decoration = {calligraphic brace, mirror,  amplitude=5pt},rotate around = {10:(2.5,2.5)}] (1.5,1.44) -- (2.9,1.44)
            node[pos=0.6, below=6pt] {$2 r_k$};
\end{tikzpicture}
\end{center}
\vspace{2mm}
%

Let $r_j^k = 2^{j-1} r_k$ and, for each $k \in \N$ and $1 \leq j \leq k$, define $Q_j^k$ to be the unique cube with radius $r_j^k$, lower corner $lc \, (Q_j^k) = lc \, (Q_1^k)$, and each vector $u_i$ is normal to one of its $n-1$ dimensional boundary hyperplanes. For a fixed $k$, the $Q_j^k$ form a sequence of increasing cubes, i.e., $Q_1^k \subset Q_2^k \subset \cdots \subset Q_k^k$. Thus, if $Q_k^k \subset D_k$, it follows that $Q_j^k \subset D_k$ for each $j$. We will prove this inclusion.

Fix a point $p \in Q_k^k$. We will show that $p \in D_k$ and hence, $Q_k^k \subset D_k$. Since $u_1, \dots , u_n$ form an orthonormal basis of $\R^n$, \\[-4mm]
%
\begin{center}
\begin{tikzpicture}[scale=0.8]]
        \draw[rotate around = {10:(5,5)}] (0,0) -- (0,6);
        \draw[rotate around = {10:(5,5)}] (0,0) -- (6,0);
        \node[] at (1.25,5) {\large $D_3$};
        \draw[rotate around = {10:(5,5)}] (0.75,0.75) rectangle (1.75,1.75)
            node[pos=0.5] {$Q^3_1$};
        \draw[rotate around = {10:(5,5)}] (0.75,0.75) rectangle (2.75,2.75)
            node[pos=0.7] {$Q^3_2$};
        \draw[rotate around = {10:(5,5)}] (0.75,0.75) rectangle (4.75,4.75)
            node[pos=0.7] {$Q^3_3$};
        \draw[rotate around = {10:(5,5)}] (0.75,0.75) circle[radius=1.5pt];
            \node[] at (1.6,-0.25) {\footnotesize{$lc (Q_1^k)$}};
        \fill[rotate around = {10:(5,5)}] (0.75,0.75) circle[radius=1.5pt];
\end{tikzpicture}
\\ \ \\
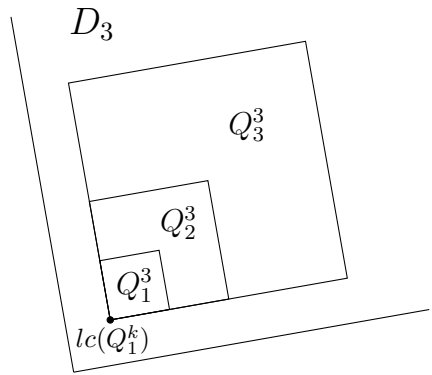
    Choice of $Q_j^k$ in $\R^2$ for $k=3$.
\end{center}
%
we may uniquely represent $p$ as $p = lc \, (Q_1^k) + s_1 u_1 + \cdots + s_n u_n$ for some $s_1, \dots , s_n \in \R$. Because $p \in Q_k^k$, we have that $0 \leq s_i \leq 2r_k^k$ for each $1 \leq i \leq n$. Notice that $2r^k_k = 2^k r_k < R_k$. Notice by our choice of the basis, for any collection of scalars $a_1 , \dots , a_n$ satisfying $0 \leq a_i \leq R_k$, we have that the point $lc \, (Q_1^k) + \sum_i a_i u_i \in D_k$. This implies that $p \in D_k$ and so $Q_j^k \subset D_k$ for each $j$. 

Now for each $Q_j^k$, let $P_j^k$ be the corresponding cube that forms a $(t,u_1)$ pair. We claim that for a fixed $k$, each cube $P_j^k$ is contained in $D_k$ and the $P_j^k$ are pairwise disjoint with respect to $j$. For example, in $\R^2$ and for $k=3$ and $t=5$, this would take the form of the following diagram: 
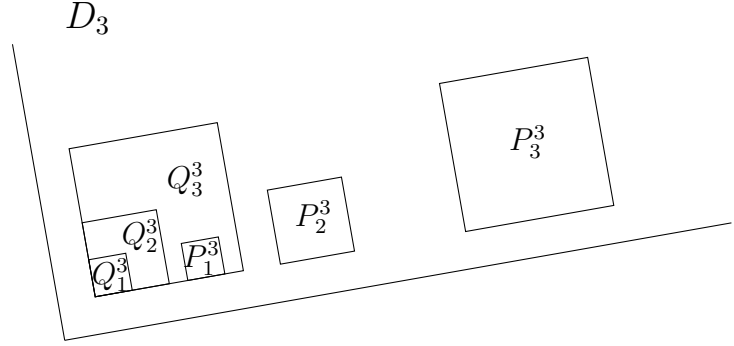
\vspace{2mm}
\begin{center}
    \begin{tikzpicture}
        \draw[rotate around = {10:(5,5)}] (0,0) -- (0,4);
        \draw[rotate around = {10:(5,5)}] (0,0) -- (9,0);
        \node[] at (1.25,3.5) {\large $D_3$};
        \draw[rotate around = {10:(5,5)}] (0.5,0.5) rectangle (1,1)
            node[pos=0.5] {$Q^3_1$};
        \draw[rotate around = {10:(5,5)}] (0.5,0.5) rectangle (1.5,1.5);
        \draw[rotate around = {10:(5,5)}] (0.5,0.5) rectangle (2.5,2.5);
        \node[] at (1.95,.6) {$Q^3_2$};
        \node[] at (2.55,1.35) {$Q^3_3$};
        \draw[rotate around = {10:(5,5)}] (1.75,0.5) rectangle (2.25,1);
        \draw[rotate around = {10:(5,5)}] (3,0.5) rectangle (4,1.5);
        \draw[rotate around = {10:(5,5)}] (5.5,0.5) rectangle (7.5,2.5);
        \node[] at (2.76,.3) {$P^3_1$};
        \node[] at (4.25,0.85) {$P^3_2$};
        \node[] at (7.1,1.85) {$P^3_3$};
    \end{tikzpicture}
\end{center}
\vspace{2mm}
We first show that the $P_j^k$ are pairwise disjoint. Choose $j_1$ and $j_2$ such that $1 \leq j_2 < j_1 \leq k$, and fix points $a \in P_{j_1}^k$ and $b \in P_{j_2}^k$.  Since  $u_1, \dots , u_n$ form an orthonormal basis, there exists scalar $v_1, \dots , v_n$ and $w_1, \dots , w_n$ such that 
$$a = lc \, (P_{j_1}^k) + \sum_i v_i u_i \quad \text{ and } \quad b = lc \, (P_{j_2}^k) + \sum_i w_i u_i.$$
Because $a \in P_{j_1}^k$, $0 \leq v_i \leq 2r_{j_1}^k$ for all $i$. Likewise, $0 \leq w_i \leq 2r_{j_2}^k$. Since $j_1 > j_2$ and $r_j^k = 2^j r_k$ for each $1 \leq j \leq k$, we have $r_{j_1}^k \geq 2r_{j_2}^k$.  Since $t > 4$, 
\begin{multline*}
    |a - b| = |(lc \, (Q_1^k) + t r_{j_1}^k \sqrt{n} u_1 + \sum_{i=1}^n v_i u_i) - (lc \, (Q_1^k) + t r_{j_2}^k \sqrt{n} + \sum_{i=1}^n w_i u_i)| \\
    = |t(r_{j_1}^k - r_{j_2}^k) \sqrt{n} u_1 + \sum_{i = 1}^n (v_i - w_i) u_i| \geq t(r_{j_1}^k - r_{j_2}^k)\sqrt{n} - |v_1 - w_1| \geq \frac{t-4}{2} r_{j_1}^k > 0.
\end{multline*}
Thus, $P_{j_1}^k$ and $P_{j_2}^k$ share no points and so are disjoint. 

We now show that each $P_j^k$ is contained in $D_k$. Fix a point $q \in P_j^k$ for some $j$, $1 \leq j \leq k$. Then there exists scalars $g_1, \dots , g_n$, $0 \leq g_i \leq 2r_j^k$, such that 
$$q = lc \, (P_j^k) + \sum_i g_i u_i = lc \, (Q_1^k) + t r_j^k \sqrt{n} u_1 + \sum_i g_i u_i.$$
But $g_i \leq 2r_j^k \leq 2^k r_k < R_k$ and 
$$t r_j^k \sqrt{n} + g_1 \leq 2^{k-1} t r_k \sqrt{n} + 2^k r_k = 2^{k-1} (t\sqrt{n}+2) r_k = R_k.$$
Since $lc \, (Q_1^k) + \sum_i a_i u_i \in D_k$ whenever $a_i \leq R_k$, we have that $q \in D_k$. From this it follows that $P_j^k \subset D_k$. \\
\indent Finally, we have that
\begin{gather} \label{Script E_k has the cube property}
    |Q^k_j \cap E_k| > \frac{1}{2} |Q^k_j| \;\;\;\; \text{ and } \;\;\;\; |P_j^k \cap E_k| > \frac{1}{2} |P_j^k|.
\end{gather}
This follows from \eqref{percent Q_k in E_k}:
\begin{multline*}
    |Q^k_j \setminus E_k| \leq |D_k \setminus E_k| 
    < 2^{-nk}(t\sqrt{n}+2)^{-n} |D_k| 
    =  2^{-nk} (t\sqrt{n}+2)^{-n} (2 R_k)^n \\
    = 2^{-nk} (t\sqrt{n}+2)^{-n} (2^{k-1} (t\sqrt{n}+2))^n (2 r_k)^n 
    = 2^{-n} (2 r_k)^n 
    \leq \frac{1}{2} |Q^k_j|.
\end{multline*}
The calculations for each $P_j^k$ are the same, using that $|Q_j^k| = |P_j^k|$. 

\medskip

Now define, for each $k \in \N$, 
\begin{gather*}
    f_k = \frac{\chi_{Q_1^k \cap E_k}}{\lVert \chi_{Q_1^k \cap E_k} \rVert_{\pp}}.
\end{gather*}
Then $\lVert f_k \rVert_{\pp} = 1$, and by Lemma \ref{Ptwise Average Estimate Lemma}, for any $y \in P_j^k \cap E_k$,
\begin{align} \label{ptwise operator estimate in M_alpha}
    M_\alpha f_k(y) &\geq \left(\frac{(t+2)\sqrt{n}}{2}\right)^{\alpha - n} |Q^k_j|^{\frac{\alpha}{n}} \dashint_{Q^k_j} f_k(x) \, dx  \nonumber \\
    &= 2^{-\alpha}\left((t+2)\sqrt{n}\right)^{\alpha - n} |Q^k_j|^{\frac{\alpha}{n}} (r^k_j)^{-n} |Q^k_1 \cap E_k| \lVert \chi_{Q^k_1 \cap E_k} \rVert_{\pp}^{-1} \nonumber \\
    &> 2^{-\alpha-n(j-1) - 1}\left((t+2)\sqrt{n}\right)^{\alpha - n}|Q^k_j|^{\frac{\alpha}{n}} r_k^{-n} |Q_1^k| \, \lVert \chi_{Q_1^k \cap E_k} \rVert_{\pp}^{-1} \nonumber \\
    &\geq 2^{-\alpha-n(j-2) - 1} \left((t+2)\sqrt{n}\right)^{\alpha - n} |P^k_j|^{\frac{\alpha}{n}} \lVert \chi_{Q_1^k \cap E_k} \rVert_{\pp}^{-1}.
\end{align}

In the following calculation we use strongly that $Q_1^k \cap E_k$ has minimal harmonic mean. Let $\mathcal{E}$ be the collection of all $Q^k_j \cap E_k$ and $P^k_j \cap E_k$, $k \in \N$ and $1 \leq j \leq k$. By \eqref{Script E_k has the cube property}, $\mathcal{E}$ has the cube property. Moreover, every set in $\mathcal{E}$ has measure less than $1$ since $|D_k| \leq 1$. Thus, by Corollary~\ref{fractional strong type gives norm equivalence corollary} (see \eqref{norm and harmonic mean equivalence eqn} for the constants), the boundedness of $M_\alpha$ implies that for any $P_j^k$,
\begin{align} \label{Q_1^k to P_j^k norm fractional}
    \lVert \chi_{Q_1^k \cap E_k} \rVert_{\pp} &\leq \frac{2K_{\pp}^2 \Kconst{\pp}{\E}}{k_{\pp}} |Q_1^k \cap E_k|^{1/p_{Q_1^k \cap E_k}}  \nonumber \\
    &\leq \frac{2K_{\pp}^2 \Kconst{\pp}{\E}}{k_{\pp}} |Q_1^k|^{1/p_{Q_1^k \cap E_k}}  \nonumber \\
    &\leq  \frac{2K_{\pp}^2 \Kconst{\pp}{\E}}{k_{\pp}} |Q_1^k|^{1/p_{P_j^k \cap E_k}}, \nonumber \\
    \intertext{where in the last inequality we use that $|Q_1^k| \leq 1$ and $p_{Q_1^k \cap E_k} \leq p_{P_j^k \cap E_k}$. Furthermore, since $1/p_{P_j^k \cap E_k} = 1/q_{P_j^k \cap E_k} + \alpha/n$, }
    &= \frac{2 K_{\pp}^2 \Kconst{\pp}{\E}}{k_{\pp}} (2^{-n(j-1)} |P^k_j|)^{1/p_{P_j^k \cap E_k}} \nonumber \\
    &\leq \frac{2^{1-n(j-1) \beta_k^{-1}}K_{\pp}^2 \Kconst{\pp}{\E}}{k_{\pp}} |P^k_j|^{\frac{\alpha}{n}} |P^k_j|^{1/q_{P^k_j \cap E_k}}. \nonumber \\
    \intertext{Finally, by \eqref{Script E_k has the cube property} and another application of Corollary \ref{fractional strong type gives norm equivalence corollary},}
    &\leq \frac{2^{2-n(j-1) \beta_k^{-1}}K_{\pp}^2 \Kconst{\pp}{\E}}{k_{\pp}} |P^k_j|^{\frac{\alpha}{n}} |P^k_j \cap E_k|^{1/q_{P^k_j \cap E_k}} \nonumber \\
    &\leq \frac{2^{3-n(j-1) \beta_k^{-1}} K_{\pp}^3 \Kconst{\pp}{\E}}{k_{\pp}} |P^k_j|^{\frac{\alpha}{n}} \lVert \chi_{P_j^k \cap E_k} \rVert_{q(\cdot)}.
\end{align}
Therefore, by \eqref{ptwise operator estimate in M_alpha} and \eqref{Q_1^k to P_j^k norm fractional}, for each $y \in P_j^k \cap E_k$ we have that
\begin{gather} \label{ptwise estimate of M_alpha}
    M_\alpha f_k(y) \geq \frac{2^{-n(j-1)(1-\beta_k^{-1})-\alpha-n-4} k_{\pp}}{\left((t+2)\sqrt{n}\right)^{n-\alpha}K_{\pp}^3 \Kconst{\pp}{\E}} \lVert \chi_{P_j^k \cap E_k} \rVert_{q(\cdot)}^{-1}.
\end{gather}
Recall that $p(x) \leq \beta_k$ on $E_k$, with $\beta_k$ being defined by $\eqref{beta_k}$. So, $q(x) \leq \frac{n\beta_k}{n-\alpha \beta_k}$ on each $E_k$, and in particular, 
\begin{gather} \label{q_+ calculation}
    q_+(D_k \cap E_k) = \frac{np_+(D_k \cap E_k)}{n - \alpha p_+(D_k \cap E_k)} \leq \frac{n\beta_1}{n-\alpha \beta_1} = \frac{n+1}{n-\alpha} < \infty.
\end{gather}  
Thus, since $P_j^k \cap E_k \subset \{x \in \R^n : q(x) < \infty\}$ and the $P_j^k$ are disjoint with respect to $j$, \eqref{ptwise estimate of M_alpha} and \eqref{q_+ calculation} together imply that for any constant $C>0$,
\begin{align*} 
    \rho_{q(\cdot)} \left(\frac{M_\alpha f_k}{C} \right) &\geq \int_{\{x \in \R^n : q(x) < \infty\}} \left(\frac{|M_\alpha f_k(x)|} {C} \right)^{q(x)} \, dx \nonumber \\
    &\geq \sum_{j=1}^{k} \int_{P_j^k \cap E_k} \left[ 
    \frac{2^{-n(j-1)(1-\beta_k^{-1})-\alpha-n-4} k_{\pp}}{C \left((t+2)\sqrt{n}\right)^{n-\alpha}K_{\pp}^3 \Kconst{\pp}{\E}} \lVert \chi_{P_j^k \cap E_k} \rVert_{q(\cdot)}^{-1} \right]^{q(x)} \, dx \nonumber \\
    &\geq \left(\frac{2^{-\alpha-n-4} \left((t+2)\sqrt{n}\right)^{\alpha-n} k_{\pp}}{C K_{\pp}^3 \Kconst{\pp}{\E}} \right)^{\left(\frac{n+1}{n-\alpha}\right)} \nonumber \\
    &\hspace{10mm} \times \sum_{j=1}^k 2^{-n(j-1)(1 - \beta_k^{-1})\left(\frac{n\beta_k}{n-\alpha \beta_k}\right)} \int_{P_j^k \cap E_k} \lVert \chi_{P_j^k \cap E_k} \rVert_{q(\cdot)}^{-q(x)} \, dx, \nonumber
\end{align*}
where in the last inequality we use that if a constant $c$ satisfies $0 < c < 1$, then $c^{q(x)} \geq c^{q_+}$, assuming $q_+ < \infty$. Furthermore, since $q(x) < \infty$ for $x \in P_j^k \cap E_k$, 
$$\int_{P_j^k \cap E_k} \lVert \chi_{P_j^k \cap E_k} \rVert_{\qq}^{-q(x)} \, dx = 1,$$ 
(see \cite[Proposition 2.21]{VLS}). Thus,
\begin{gather} \label{First Step Modular Estimate of M_alpha f_k}
    \rho_{\qq} \left(\frac{M_\alpha f_k}{C}\right) \geq \left(\frac{2^{-\alpha-n-4} \left((t+2)\sqrt{n}\right)^{\alpha-n} k_{\pp}}{C K_{\pp}^3 \Kconst{\pp}{\E}} \right)^{\left(\frac{n+1}{n-\alpha}\right)} \sum_{j=1}^k 2^{-nk(1-\beta_k^{-1})\left(\frac{n\beta_k}{n-\alpha \beta_k}\right)}. 
\end{gather}
By our choice of $\beta_k$,
\begin{gather} \label{exponent control in modular calculation}
    -nk(1-\beta_k^{-1})\left(\frac{n\beta_k}{n-\alpha \beta_k}\right) = -nk\left(\frac{n-\alpha}{n^2k + n}\right) \left(\frac{nk+1}{(n-\alpha)k} \right) = -\frac{n^2k^2 + nk}{n^2k^2+n k} = -1.
\end{gather}
If we combine \eqref{First Step Modular Estimate of M_alpha f_k} and \eqref{exponent control in modular calculation}, we obtain
\begin{multline} \label{Modular Estimate of M_alpha f_k}
    \rho_{\qq} \left(\frac{M_\alpha f_k}{C}\right) 
    \geq \left(\frac{2^{-\alpha-n-4} \left((t+2)\sqrt{n}\right)^{\alpha-n} k_{\pp}}{C K_{\pp}^3 \Kconst{\pp}{\E}} \right)^{\left(\frac{n+1}{n-\alpha}\right)} \sum_{j=1}^k 2^{-1}  \\
    \geq k \left(\frac{2^{-\alpha-n-5} \left((t+2)\sqrt{n}\right)^{\alpha-n} k_{\pp}}{C K_{\pp}^3 \Kconst{\pp}{\E}} \right)^{\left(\frac{n+1}{n-\alpha}\right)}.
\end{multline}

Now let $C=\|M_\alpha\|_{L^\pp \rightarrow L^\qq}$.  By the definition of the norm, since $\lVert \frac{M_\alpha f_k}{C} \rVert_{\qq} \leq \lVert f_k \rVert_{\pp} = 1$, we have that $\rho_{\qq} (\frac{M_\alpha f_k}{C}) \leq 1$ for all $k \in \mathbb{N}$. But for
\[ k > \left(\frac{2^{-\alpha-n-5} \left((t+2)\sqrt{n}\right)^{\alpha-n} k_{\pp}}{C K_{\pp}^3 \Kconst{\pp}{\E}}\right)^{\left(\frac{n+1}{\alpha-n}\right)}, \]
\eqref{Modular Estimate of M_alpha f_k} implies $\rho_{\qq} (\frac{M_\alpha f_k}{C}) > 1$. This is a contradiction, and hence, if $M_\alpha : L^{\pp} \to L^{\qq}$, then $p_- > 1$.
\end{proof}
\section{Fractional Singular Integral Operators} \label{Fractional Singular Integral Operators}
In this section we prove Theorems~\ref{SIO K_0^alpha Theorem} and~\ref{SIO Bounded Exponent Theorem}.  To do so, we will show that given a fractional CZO $T_\alpha$ that is non-degenerate in the direction $u_1$, there exists $t$ sufficiently large such that $T_\alpha$ behaves pointwise like the fractional maximal operator on $(t,u_1)$ pairs. With this we can apply many of the same arguments in Section \ref{Fractional Maximal Operators}.  As we noted above, this is the reason that we wrote the proofs for the fractional maximal operator as we did.

The following lemma was proved in \cite[Theorem 2.1]{JOLinear} in the case $\alpha = 0$. The proof for $0 \leq \alpha < n$ is very similar, and we provide it for completeness.
\begin{lemma} \label{Ptwise Fractional SIO Average Estimate Lemma}
    Given $0 \leq \alpha < n$, let $T_\alpha$ be a fractional Calder\'on-Zygmund operator with a non-degenerate kernel in the direction $u_1$. Let $Q$ and $P$ be cubes of radius $r > 0$ that form a $(t,u_1)$ pair for some $t \in \R$. Then there exists $t_0 \geq 4$ such that if $|t| \geq t_0$, $y \in P$, and $f$ is a non-negative function with $\supp(f) \subset Q$, we have that 
    \begin{gather}
        |T_\alpha f(y)| \geq \frac{2^{n-\alpha-1} a}{(t\sqrt{n})^{n-\alpha}} |Q|^{\frac{\alpha}{n}} \dashint_{Q} f(x) \, dx,
    \end{gather}
    where $a$ is the constant in the non-degeneracy condition \eqref{non-degenerate condition}.
\end{lemma}
\begin{proof}
    Choose a constant $t_0 \geq 4$ such that $2C_0 (1+2^{n-\alpha+\delta})t_0^{-\delta} \leq a$ and let $|t| \geq t_0$. Here, $a$ is the constant in \eqref{non-degenerate condition}, and $\delta, C_0$ are the constants in~\eqref{kernel shift bound}. Let $Q = Q(x_0,r)$ and $P = Q(y_0,r)$ form a $(t,u_1)$ pair (Definition \ref{QP Pair Definition}). Given any point $x \in Q$ we can write $x = x_0 + h$, where $|h| < r \sqrt{n}$. Similarly, given $y \in P$, $y = y_0 + k$ where $|k| < r \sqrt{n}$. We claim that for such $x$ and $y$,
    \begin{gather} \label{Kernel Estimate}
        |K_\alpha(x,y) - K_\alpha(x_0,y_0)| \leq \frac{1}{2} |K_\alpha(x_0,y_0)|.
    \end{gather}
    To prove this we will apply \eqref{kernel shift bound}, which is possible as $|h| < r \sqrt{n} \leq \frac{1}{2} |x_0 - y_0|$ and
    \begin{gather*}
        |x_0 + h - y_0| \geq |x_0 - y_0| - |h| \geq |t|r \sqrt{n} - r \sqrt{n} \geq \frac{|t|}{2} r \sqrt{n} \geq 2|k|.
    \end{gather*}
    Thus, we can estimate as follows:
    \begin{align*}
        &  |K_\alpha(x,y) - K_\alpha(x_0,y_0)| \\
        & \qquad \leq |K_\alpha(x_0 + h, y_0 + k) - K_\alpha(x_0 + h,y_0)| 
        + |K_\alpha(x_0 + h,y_0) - K_\alpha(x_0,y_0)| \\
        & \qquad \leq \frac{C_0 |k|^{\delta}}{|x_0 + h - y_0|^{n-\alpha+\delta}} + \frac{C_0 |h|^{\delta}}{|x_0-y_0|^{n-\alpha+\delta}} \\
        & \qquad = I_1 + I_2.
    \end{align*}
    We can bound $I_2$ immediately:
    \begin{gather*}
        I_2 \leq \frac{C_0(r \sqrt{n})^{\delta}}{(|t|r \sqrt{n})^{\delta} |x_0 - y_0|^{n-\alpha}} = C_0 \frac{|t|^{-\delta}}{|x_0 - y_0|^{n-\alpha}}.
    \end{gather*}
    To estimate $I_1$, note that 
    \begin{gather*}
        |x_0 + h - y_0| \geq \frac{|t|}{2} r \sqrt{n} = \frac{1}{2} |x_0 - y_0|.
    \end{gather*}
    Hence,
    \begin{gather*}
        I_1 \leq \frac{C_0 2^{n-\alpha+\delta} (r \sqrt{n})^{\delta}}{(|t|r \sqrt{n})^{\delta} |x_0 - y_0|^{n-\alpha}} = C_0 \frac{2^{n-\alpha+\delta} |t|^{-\delta}}{|x_0 - y_0|^{n-\alpha}}.
    \end{gather*}
    If we combine these estimates, by our choice of $t$ and \eqref{non-degenerate condition} we have
    \begin{gather*}
        I_1 + I_2 \leq \frac{a}{2} \frac{1}{|x_0 - y_0|^{n-\alpha}} \leq \frac{1}{2} |K_\alpha(x_0,y_0)|,
    \end{gather*}
    which proves \eqref{Kernel Estimate}. \\
    \indent It now follows that for any $x \in Q$, $y \in P$, the kernel $K_\alpha(x,y)$ always has the same sign. Therefore, if we fix a non-negative function $f$ with $\supp(f) \subset Q$, then
    \begin{align*}
        |T_\alpha f(y)| &= \left|\int_Q K_\alpha(x,y)f(x) \, dx\right| \\
        &= \int_Q |K_\alpha(x,y)|f(x) \, dx \\
        &\geq \int_Q |K_\alpha(x_0,y_0)| f(x) \, dx \\
        &\qquad  - \int_Q |K_\alpha(x,y) - K_\alpha(x_0,y_0)|f(x) \, dx;\\
    \intertext{again by \eqref{Kernel Estimate} and \eqref{non-degenerate condition},}
        &\geq \frac{1}{2} |K_\alpha(x_0,y_0)| \int_Q f(x) \, dx \\
        &\geq \frac{a}{2|x_0 - y_0|^{n-\alpha}} \int_Q f(x) \, dx \\
        &\geq \frac{a}{2(|t|r\sqrt{n})^{n-\alpha}} \int_Q f(x) \, dx \\
        &= \frac{2^{n-\alpha-1} a}{(|t| \sqrt{n})^{n-\alpha}} |Q|^{\frac{\alpha}{n}} \dashint_Q f(x) \, dx.
    \end{align*}
\end{proof}
\begin{proof}[Proof of Theorem \ref{SIO K_0^alpha Theorem}]
We can argue exactly as in the proof of Theorem~\ref{Maximal K_0^alpha theorem}.  To show that $\pp \in K_0^\alpha(\R^n)$ whenever $T_\alpha : L^{\pp}(\R^n) \to L^{\qq,\infty}(\R^n)$, we will prove the fractional averaging operators are uniformly bounded since, by Proposition \ref{K_0^alpha and A_Q^alpha proposition}, this is equivalent to $\pp \in K_0^\alpha(\R^n)$.  Let $|t| \geq t_0$, where $t_0$ is the constant in Lemma \ref{Ptwise Fractional SIO Average Estimate Lemma}. Now let $Q = Q(x_0,r)$ and $P = Q(y_0,r)$ form a $(t,u_1)$ pair. By Lemma \ref{Ptwise Fractional SIO Average Estimate Lemma}, for a non-negative function $f$ with $\supp (f) \subset Q$ and $y \in P$,
\begin{gather*}
    |T_\alpha f(y)| \geq c(a,t,n,\alpha) |Q|^{\frac{\alpha}{n}} \dashint_{Q} f(x) \, dx.
\end{gather*}
It follows that $P \subset \{x \in \R^n : |T_\alpha f(y)| > \lambda\}$ for any $0 < \lambda < c(a,t,n,\alpha)|Q|^{\frac{\alpha}{n}} \dashint_{Q} f\,dx$. Since $T_\alpha: L^{\pp} \to L^{\qq,\infty}$, we have
\begin{gather*}
    \lVert \chi_{P} \rVert_{\qq} \leq \lVert \chi_{\{|T_\alpha f(y)| > \lambda\}} \rVert_{\qq} 
    \leq \frac{\|T_\alpha\|_{\Lp \rightarrow L^{\qq,\infty}}}{\lambda} \lVert f \rVert_{\pp}.
\end{gather*}
If we take the supremum over all such $\lambda$, we obtain
\begin{gather} \label{Supremum on lambda estimate} 
    \left( \dashint_{Q} f\,dx \right) |Q|^{\frac{\alpha}{n}} \lVert \chi_{P} \rVert_{\qq} \leq c(a,t,n,\alpha) \lVert f \rVert_{\pp}.
\end{gather}
By \eqref{Supremum on lambda estimate} and the generalized H\"older's inequality (Corollary \ref{gen Holder}),
\begin{gather} \label{disjoint cube bound}
    \left( \dashint_{Q} f\,dx \right) \lVert \chi_{P} \rVert_{\qq} \lesssim |Q|^{-\frac{\alpha}{n}} \lVert f \rVert_{\pp} \lesssim |Q|^{-\frac{\alpha}{n}} \lVert \chi_{Q} \rVert_{\frac{n}{\alpha}} \lVert f \rVert_{\qq} = \lVert f \rVert_{\qq} .
\end{gather}
Let $f = \chi_{Q}$; then \eqref{disjoint cube bound} yields $\lVert \chi_{P} \rVert_{\qq} \leq c(\pp,a,t,n,\alpha) \lVert \chi_{Q} \rVert_{\qq}$. Similarly, we can repeat the above argument, exchanging the roles of $Q$ and $P$ and defining $f = \chi_{P}$, to obtain the reverse inequality, i.e. $\lVert \chi_{Q} \rVert_{\qq} \leq c(\pp,a,t,n,\alpha) \lVert \chi_{P} \rVert_{\qq}$. With this we may conclude by \eqref{Supremum on lambda estimate} that
\begin{multline*}
    \lVert A_{Q}^\alpha f \rVert_{\qq} = \left( \dashint_{Q} f\,dx \right) |Q|^{\frac{\alpha}{n}} \lVert \chi_{Q} \rVert_{\qq} \\
    \leq c(\pp,a,t,n,\alpha) \left( \dashint_{Q} f\,dx \right)|Q|^{\frac{\alpha}{n}} \lVert \chi_{P} \rVert_{\qq} \leq c(\pp,a,t,n,\alpha) \lVert f \rVert_{\pp}.
\end{multline*}
Since this constant does not depend on the choice of $Q$, we have that the averaging operators are uniformly bounded from $L^{\pp}(\R^n) \to L^{\qq}(\R^n)$, and hence, by Proposition \ref{K_0^alpha and A_Q^alpha proposition}, we are done.
\end{proof}
Our next proposition shows that for a general collection of sets $\E$ with $0 < |E| < \infty$ for all $E \in \E$, if $\E$ satisfies the cube property (Definition \ref{cube property}), then $\pp \in K_0^\alpha(\E)$. The proof is identical to that of Proposition \ref{Generalized K_0^alpha Theorem}.
\begin{prop} \label{Generalized K_0^alpha theorem fractional SIO}
     Given  $0 \leq \alpha < n$ and $\pp \in \mathcal{P}(\R^n)$ such that $p_+<\frac{n}{\alpha}$, define $q(\cdot)$ by~\eqref{q dot defn}. Let $T_\alpha$ be a fractional Calder\'on-Zygmund operator with a non-degenerate kernel in the direction $u_1$.  If $\mathcal{E}$ has the cube property and   $T_\alpha : L^{\pp}(\R^n) \to L^{q(\cdot),\infty}(\R^n)$, then $\pp \in K_0^\alpha(\mathcal{E})$.  
\end{prop}
\begin{corollary} \label{fractional SIO strong type gives norm equivalence corollary}
   Given  $0 \leq \alpha < n$ and $\pp \in \mathcal{P}(\R^n)$ such that $p_+<\frac{n}{\alpha}$, define $q(\cdot)$ by~\eqref{q dot defn}.  Let $T_\alpha$ be a fractional Calder\'on-Zygmund operator with a non-degenerate kernel in the direction $u_1$.  If  $\mathcal{E}$ have the cube property and  $T_\alpha : L^{\pp}(\R^n) \to L^{q(\cdot)}(\R^n)$,  then for all $E \in \mathcal{E}$,
    \begin{gather*}
        \lVert \chi_E \rVert_{\pp} \approx |E|^{\frac{1}{p_E}} \hspace{4mm} , \hspace{4mm} \lVert \chi_E \rVert_{p'(\cdot)} \approx |E|^{\frac{1}{p'_E}}, \\
        \lVert \chi_E \rVert_{q(\cdot)} \approx |E|^{\frac{1}{q_E}} \hspace{4mm} \text{ and } \hspace{4mm} \lVert \chi_E \rVert_{q'(\cdot)} \approx |E|^{\frac{1}{q'_E}}.
    \end{gather*} 
\end{corollary}
\begin{proof}
    This follows at once from Proposition \ref{Generalized K_0^alpha theorem fractional SIO} and Corollary \ref{K_0^alpha implies norm equivalence}.
\end{proof}
\begin{proof}[Proof of Theorem \ref{SIO Bounded Exponent Theorem}]
    By  Lemma \ref{Ptwise Fractional SIO Average Estimate Lemma} and Lemma \ref{Ptwise Average Estimate Lemma}, we have that for a sufficiently large choice of $|t|$, $T_\alpha$ and $M_\alpha$ have essentially the same  pointwise behavior on $(t,u_1)$ pairs. Hence, to prove that we cannot have $p_- = 1$, we can repeat the construction given in Theorem \ref{Maximal Bounded Exponent Theorem}. If we define 
    \[ f_k = \frac{\chi_{Q_1^k \cap E_k}}{\lVert \chi_{Q_1^k \cap E_k} \rVert_{\pp}}, \]
    then we obtain the pointwise estimate 
    \[ |T_\alpha f_k(y)| \gtrsim 2^{-n(j-1)(1-\beta_k^{-1})} \lVert \chi_{P_j^k \cap E_k} \rVert_{\qq}^{-1} \]
    for $y \in P_j^k \cap E_k$ as before. From this we conclude that
    \begin{gather} \label{modular estimate for T_alpha}
        \rho_{\qq}\left(\frac{T_\alpha f_k}{C} \right) \gtrsim k.
    \end{gather}
    Since $\lVert f_k \rVert_{\pp} = 1$, this contradicts the assumption $\lVert T_\alpha f_k \rVert_{\qq} \leq C \lVert f_k \rVert_{\pp}$ for sufficiently large $k$. 

    \medskip
    
    We will now prove that we cannot have  $p_+ = \frac{n}{\alpha}$; we will again use a proof by contradiction, and we will argue by duality. Suppose $p_+ = \frac{n}{\alpha}$; since $T_\alpha$ is bounded from $L^{\pp} \to L^{\qq}$, its dual operator $T_\alpha^*$ is bounded from $(L^{\qq})^* \to (L^{\pp})^*$. Moreover, we have that the kernel of $T_\alpha^*$ satisfies $K_\alpha^*(x,y) = K_\alpha(y,x)$ and so it satisfies \eqref{kernel singularity bound} and \eqref{kernel shift bound} in Definition \ref{CZ operator}. Finally, $K_\alpha^*(x,y)$ is non-degenerate in the same direction $u_1$ as $T_\alpha$ from the symmetry of \eqref{non-degenerate condition}. Therefore, $T_\alpha^*$ is also a nondegenerate CZO.  Hence, we may apply Lemma ~ \ref{Ptwise Average Estimate Lemma} to $T_\alpha^*$ to obtain a pointwise estimate of $T_\alpha^*$ over $(t,u_1)$ pairs. 
    
    Given this estimate, since $1<(p')_-  \leq (p')_+ <\infty$ and  $(q')_- = 1$, we can then repeat our construction above with $T_\alpha^*$ replacing $T_\alpha$,  $q'(\cdot)$ replacing $\pp$, and  $p'(\cdot)$ replacing $\qq$. If we define $f_k = \chi_{Q_1^k \cap E_k} \lVert \chi_{Q_1^k \cap E_k} \rVert_{q'(\cdot)}^{-1}$ as before, then the same argument yields the analog of \eqref{modular estimate for T_alpha} for $T_\alpha^*$ and $p'(\cdot)$:  for $x\in D_k$,
\begin{gather*}
    \rho_{p'(\cdot)}({T_\alpha^* f_k}) \gtrsim k.     
\end{gather*}
Since $\lVert f_k \rVert_{q'(\cdot)} = 1$ and $(p')_+(D_k) = \frac{n(q')_+(D_k)}{n-\alpha (q')_+(D_k)} \leq \frac{n\beta_k}{n-\alpha \beta_k} \leq \frac{2n}{n-\alpha}$ for $k$ sufficiently large, by a straightforward calculation it follows that
\begin{gather} \label{norm estimate for T^*}
    \lVert f_k \rVert_{L^{q'(\cdot)}(D_k)} \lesssim k^{\frac{\alpha - n}{2n}} \, \lVert T_\alpha^* f_k \rVert_{L^{p'(\cdot)}(D_k)}.
\end{gather}
By Theorem \ref{associate norm}, there exists a non-negative function $h \in L^{\pp}(D_k), \lVert h \rVert_{L^{\pp}(D_k)} \leq 1$, such that
\begin{gather} \label{norm of T^* attained via h}
    \lVert T_\alpha^* f_k \rVert_{L^{p'(\cdot)}(D_k)} \leq k_{\pp}^{-1} \int_{D_k} T_\alpha^* f_k(x) h(x) \, dx. 
\end{gather} 
By Fatou's lemma we may assume $h$ is bounded, since
\[ \int_{D_k} T_\alpha^* f_k(x) h(x) \, dx \leq \liminf_{N \to \infty} \int_{D_k} T_\alpha^* f_k(x) \min (N,h(x)) \, dx.\]
Let $1 < p \leq q < \infty$ be the fixed exponents in Definition \ref{CZ operator}. Then as $h$ is a bounded function with $\supp{(h)} \subset D_k$, $h \in L^{p}(\R^n)$. Likewise, $f_k \in L^{q'}(\R^n)$. Hence by a duality argument on the classical Lebesgue spaces,
\begin{gather} \label{Exchanging T^* f_k for T h}
    \int_{D_k} T_\alpha^* f_k(x) h(x) \, dx = \int_{D_k} f_k(x)\, T_\alpha h(x) \, dx.
\end{gather}  
By H\"older's inequality in the variable Lebesgue spaces, we have
\begin{gather} \label{Splitting up f_k Th integral}
    \int_{D_k} f_k(x) T_\alpha h(x) \, dx \leq K_{\qq} \lVert f_k \rVert_{L^{q'(\cdot)}(D_k)} \, \lVert T_\alpha h \rVert_{L^{\qq}(D_k)}.
\end{gather}
Thus, if we combine inequalities \eqref{norm estimate for T^*}--\eqref{Splitting up f_k Th integral}, and use the assumption  that $T_\alpha : L^{\pp} \to L^{\qq}$, we get
\begin{multline*}
    \lVert f_k \rVert_{L^{q'(\cdot)}(D_k)} \lesssim k^{\frac{\alpha-n}{2n}} \lVert f_k \rVert_{L^{q'(\cdot)}(D_k)} \lVert T_\alpha h \rVert_{L^{\qq}(D_k)} \\ 
    \lesssim k^{\frac{\alpha-n}{2n}} \lVert f_k \rVert_{L^{q'(\cdot)}(D_k)} \, \lVert h \rVert_{L^{\pp}(D_k)} \leq k^{\frac{\alpha-n}{2n}} \lVert f_k \rVert_{L^{q'(\cdot)}(D_k)}.
\end{multline*}
Since $\lVert f_k \rVert_{L^{q'(\cdot)}(D_k)} > 0$, this is a contradiction for $k$ sufficiently large. Therefore, we must have that $p_+ < \tfrac{n}{\alpha}$
and our proof is complete.
\end{proof}
\section{More on the $K_0^\alpha$ Condition} \label{Necessary but Not Sufficient}

In this section we construct four examples related to the  $K_0^\alpha$ condition. We first state each example and make some remarks about them, and then give their constructions.  

Our first example shows that the $K_0^\alpha$ condition, while necessary, is not sufficient for the fractional maximal operator to satisfy a strong $(\pp,\qq)$ inequality.  This complements and extends~\cite[Example~4.51]{VLS} which showed that the $K_0$ condition is not sufficient for the Hardy-Littlewood maximal operator to be bounded on $\Lp$.  

\begin{example} \label{example:K0alpha-not-suff}
 Let $n=1$ and fix $0<\alpha<\frac{1}{2}$.  Then there exists an exponent function $\pp \in K_0^\alpha(\R)$, $1 < p_- \leq p_+ < \tfrac{1}{\alpha}$, such that $M_\alpha : L^{\pp} \not \to L^{\qq}$.
 \end{example}

 \begin{remark}
    The restriction that  $n = 1$ is to simplify the construction.  However, the restriction  $0 < \alpha < \tfrac12$ is  necessary for our particular construction since it relies on the fact that when $0 < \alpha < \tfrac12$ we have $\frac{1}{1-\alpha} \leq \frac{1}{\alpha}$. We strongly believe that a similar example should exist in the range $\tfrac12 \leq \alpha < 1$, but have not been able to find one.
\end{remark}

\begin{remark}
An immediate consequence of Example~\ref{example:K0alpha-not-suff} is that  the $K_0^\alpha(\R)$  is not sufficient for the Riesz potential $I_\alpha$ to satisfy a strong $(\pp,\qq)$ inequality.  This follows because we have that $M_\alpha f(x) \leq cI_\alpha f(x)$ (see~\cite{MR3642364}).   It would be interesting to see if the construction of this example could be modified to prove that the $K_0^\alpha$ condition is not sufficient for a nondegenerate fractional CZO to satisfy a strong $(\pp,\qq)$ inequality, perhaps by using Lemma~\ref{Ptwise Fractional SIO Average Estimate Lemma}.
\end{remark}

\begin{remark}
We note in passing that the $K_0^\alpha$ condition is strictly weaker than log-H\"older continuity.  From Remark~\ref{remark:extrapol} we know that there exists a discontinuous exponent $\pp$ such that $M_\alpha : \Lp \rightarrow L^\qq$; but then, by Theorem~\ref{Maximal K_0^alpha theorem},  this exponent satisfies $\pp \in K_0^\alpha$.  
\end{remark}

\medskip

Our second example is related to Proposition~\ref{norm and harmonic mean proposition}.  For cubes, it is well known (see~\cite[Theorem~4.5.7]{Diening}) that we have the following equivalence:  $\pp \in K_0(\rn)$ if and only if $\|\chi_Q\|_\pp \approx |Q|^{\frac{1}{p_Q}}$ and $\|\chi_Q\|_\cpp \approx |Q|^{\frac{1}{p_Q'}}$.  This leaves open the question as to whether either of these latter two conditions is itself sufficient to imply the $K_0$ condition.  The answer is no.

\begin{example} \label{example:K0stronger}
    There exists $\pp \in \mathcal{P}(\R)$, $1 < p_- \leq p_+ < \infty$, such that $\|\chi_I\|_{\pp} \simeq |I|^{1/p_I}$ for all intervals $I \subset \R$ but $\pp \not \in K_0(\R)$.
\end{example}    

\begin{remark}
    It is immediate that by exchanging the roles of $\pp$ and $p'(\cdot)$ we also obtain an example such that $\| \chi_Q \|_{p'(\cdot)} \simeq |Q|^{1/p'_Q}$ holds for all intervals but $\pp \not \in K_0(\R^n)$.
\end{remark}

\medskip

Our final two examples are concerned with the relationship between the $K_0$ and $K_0^\alpha$ conditions.  By Lemma~\ref{K_0^alpha implies K_0 lemma}, for $0\leq \alpha < n$, we have that $\pp\in K_0^\alpha(\rn)$ if and only if $\pp,\, \qq \in K_0(\rn)$.   However, the assumption that one of either $\pp$ or $\qq$ satisfies the $K_0$ condition is not sufficient to imply that $\pp$ satisfies the $K_0^\alpha$ condition.  

\begin{example} \label{example:ppK0-not K0alpha}
    Fix $0 < \alpha < 1$. There exists $\pp \in K_0(\R)$, $1 < p_- \leq p_+ < \frac{1}{\alpha}$, such that $\pp \not \in K_0^\alpha(\R)$.
    \end{example}

    \begin{example} \label{example:qqK0-notK0alpha}
    Fix $0 < \alpha < 1$. There exists $\pp \in \mathcal{P}(\R)$, $1 < p_- \leq p_+ < \frac1\alpha$, such that $\qq \in K_0(\R)$ but $\pp \not \in K_0^\alpha(\R)$.
    \end{example}
    
   \medskip

   We now give the constructions of each example.
   
    \begin{proof}[Construction of Example~\ref{example:K0alpha-not-suff}]
        Our example is adapted from~\cite[Example 4.51]{VLS}, which shows that the $K_0(\R^n)$ condition is not sufficient for the Hardy-Littlewood maximal operator to be bounded. Let $n = 1$ and fix $0 < \alpha < \tfrac12$. Let $\varphi \in C_c^\infty(\R)$ be such that $\supp(\varphi) \subset [-1/2,1/2]$, $0 \leq \varphi(x) \leq \frac{5 - 8\alpha + 5\alpha^2}{6\alpha(1-\alpha)(2-\alpha)}$, and $\varphi(x) = \frac{5-8\alpha+5\alpha^2}{6\alpha(1-\alpha)(2-\alpha)}$ if $x \in [-1/4,1/4]$. Define the exponent function $\pp \in \mathcal{P}(\R)$ by
    \[ p(x) = \frac{1+\alpha}{2\alpha(2-\alpha)} + \sum_{k=1}^\infty \varphi(x - e^k). \]
    Since $0 < \alpha < \tfrac12$, $1 < \frac{1+\alpha}{2\alpha(2-\alpha)} \leq p(x) \leq \frac{2-\alpha}{3\alpha(1-\alpha)} < \frac{1}{\alpha}$.  Further, $\pp \in LH_0(\R)$ (see Definition~\ref{defn-LH}). Moreover, since $p_+ < \infty$, \cite[Proposition 2.3]{VLS} gives $1/\pp \in LH_0(\R)$. Since $\frac{1}{q(x)} = \frac{1}{p(x)} - \alpha$,
    $1/\qq \in LH_0(\R)$ as well. Again by \cite[Proposition 2.3]{VLS} and the fact that $(1/\qq)_+ < \infty$, $\qq \in LH_0(\R)$. 
    
    We will first show that $\pp \in K_0^\alpha(\R)$. Fix an interval $Q$. Suppose first that  $|Q| \leq 1$. Then by \cite[Corollary~2.23]{VLS}, 
    \begin{equation} \label{eqn:smallcube1} 
    \lVert \chi_Q \rVert_{p'(\cdot)} \lVert \chi_Q \rVert_{\qq} 
    \leq |Q|^{\frac{1}{(p')_+(Q)} + \frac{1}{q_+(Q)}} 
    = |Q|^{1 - \frac{1}{p_-(Q)} + \frac{1}{q_+(Q)}}. 
    \end{equation}
    Since $1/\qq \in LH_0(\R)$ and the fact that $(1/\qq)_+ < \infty$,  by \cite[Lemma 3.24]{VLS} there exists $C > 0$ such that for all $x \in Q$,
    \begin{multline} \label{eqn:smallcube2}
         |Q|^{1 - \frac{1}{p_-(Q)} + \frac{1}{q_+(Q)}} \\
        = |Q|^{1 - \alpha - \frac{1}{q_-(Q)} + \frac{1}{q_+(Q)}} 
        |Q|^{1-\alpha + \left(\left(\frac{1}{\qq}\right)_-(Q) - \frac{1}{q(x)}\right) + \left(\frac{1}{q(x)} - \left(\frac{1}{\qq}\right)_+(Q)\right)} \leq C|Q|^{1 - \alpha}.
    \end{multline}
    If we combine these inequalities we get that the $K_0^\alpha$ condition holds.
    
    Now suppose $|Q| > 1$. Then there exists $j \in \N$ such that $e^{j-1} < |Q| \leq e^j$. For each $k \in \N$, define the intervals $A_k = [-1/2 + e^k, 1/2 + e^k]$. Since $\rho_{p'(\cdot)}(\chi_{A_k}) = \rho_{\qq}(\chi_{A_k}) = |A_k| = 1$, by \cite[Corollary 2.23]{VLS} we have $\lVert \chi_{A_k} \rVert_{p'(\cdot)}, \lVert \chi_{A_k} \rVert_{\qq} \leq 1$ for all $k \in \N$. The interval $Q$ can intersect at most $j_0 \leq j$ consecutive intervals, say $A_{k_1} , \dots , A_{k_{j_0}}$. Let $P = Q \setminus \bigcup_{i=1}^{j_0} A_{k_i}$. For $x \in P$, $p'(x) = \frac{1+\alpha}{2\alpha^2-3\alpha+1}$. 
    
    We claim there exists $C(\alpha) > 0$ such that for all $j \geq 1$,
    $$j \leq C(\alpha) e^{(j-1)(2\alpha^2-3\alpha+1)/(1+\alpha)}.$$
    Define the functions $y_1(x) = x$ and $y_2(x) = e^{(x-1)(2\alpha^2-3\alpha+1)/(1+\alpha)}$. Then $y_1'(x) = 1$ and $y_2'(x) = \frac{2\alpha^2-3\alpha+1}{1+\alpha}e^{(x-1)(2\alpha^2-3\alpha+1)/(1+\alpha)} \geq \frac{2\alpha^2-3\alpha+1}{1+\alpha} $ for all $x \geq 1$. Thus $\frac{1+\alpha}{2\alpha^2-3\alpha+1} y_2(1) \geq y_1(1)$ and $\frac{1+\alpha}{2\alpha^2-3\alpha+1}y_2'(x) \geq y_1'(x)$ for all $x \geq 1$. This implies that for $x \geq 1$, 
    $$x \leq \frac{1+\alpha}{2\alpha^2-3\alpha+1} e^{(x-1)(2\alpha^2-3\alpha+1)/(1+\alpha)}.$$

    We now use this inequality to estimate as follows:
    \begin{multline*} \lVert \chi_Q \rVert_{p'(\cdot)} \leq \sum_{i=1}^{j_0} \lVert \chi_{A_{k_i}} \rVert_{p'(\cdot)} + \lVert \chi_P \rVert_{p'(\cdot)} \\
    = j_0 + |P|^{\frac{2\alpha^2-3\alpha+1}{1+\alpha}} \leq j + |Q|^{\frac{2\alpha^2-3\alpha+1}{1+\alpha}} \leq C(\alpha) |Q|^{\frac{2\alpha^2-3\alpha+1}{1+\alpha}}. 
    \end{multline*}

    We argue similarly for $\lVert \chi_{A_k} \rVert_{\qq}$. 
    We have that $q(x) = \frac{1+\alpha}{3\alpha(1-\alpha)}$ for all $x \in P$. A similar computation as above shows $j \leq C(\alpha) e^{3(j-1)(\alpha(1-\alpha))/(1+\alpha)}$ for all $j \geq 1$. Hence,
    \[ \lVert \chi_Q \rVert_{\qq} \leq \sum_{i=1}^{j_0} \lVert \chi_{A_{k_i}} \rVert_{\qq} + \lVert \chi_P \rVert_{\qq} = j_0 + |P|^{\frac{3\alpha(1-\alpha)}{1+\alpha}} \leq j + |Q|^{\frac{3\alpha(1-\alpha)}{1+\alpha}} \leq C(\alpha) |Q|^{\frac{3\alpha(1-\alpha)}{1+\alpha}}. \]
  If we combine these two estimates, we get
    \[ \lVert \chi_Q \rVert_{p'(\cdot)} \lVert \chi_Q \rVert_{\qq} \leq C(\alpha) |Q|^{\frac{2\alpha^2-3\alpha+1}{1+\alpha} + \frac{3\alpha(1-\alpha)}{1+\alpha}} = C(\alpha) |Q|^{1-\alpha}; \]
    hence, $\pp \in K_0^\alpha(\R)$. 

    \medskip

    \indent We now show that $M_\alpha$ is not bounded from $L^{\pp}(\R) \to L^{\qq}(\R)$. Define the sets $B_k = [-1/4 + e^k, 1/4 + e^k]$ and $C_k = [-3/2 + e^k, 3/2 + e^k] \setminus A_k$, and let
    \[ f(x) = \sum_{k=1}^\infty k^{-\frac{3\alpha(1-\alpha)}{1+\alpha}} \chi_{B_k}. \]
    Since
    \[ p(x) = \frac{1+\alpha}{2\alpha(2-\alpha)} + \frac{5-8\alpha+5\alpha^2}{6\alpha(1-\alpha)(2-\alpha)} = \frac{2-\alpha}{3\alpha(1-\alpha)}, \]
    for all $x \in B_k$,
    \begin{gather} \label{example modular}
        \rho_{\pp}(f) = \sum_{k=1}^\infty \int_{B_k} k^{-\frac{3p(x)\alpha(1-\alpha)}{1+\alpha}} \, dx = \frac{1}{2} \sum_{k=1}^\infty k^{-\frac{2-\alpha}{1+\alpha}} < \infty. 
    \end{gather} 
    Here we again use that $0 < \alpha < \tfrac12$: this implies that $-\frac{2-\alpha}{1+\alpha} < -1$ and so the sum  converges. Since $p_+ < \infty$, \eqref{example modular} implies $f \in L^{\pp}(\R)$ by \cite[Proposition 2.12]{VLS}. 
    
    On the other hand, if $x \in C_k$,
    \[ M_\alpha f(x) \geq |A_k \cup C_k|^{\alpha-1} \int_{A_k \cup C_k} f(y) \, dy = 3^{\alpha-1} \sum_{k=1}^\infty \int_{B_k} f(y) \, dy = \frac{3^{\alpha-1}}{2} k^{-\frac{3\alpha(1-\alpha)}{1+\alpha}}.  \]
    Since $q(x) = \frac{1+\alpha}{3\alpha(1-\alpha)}$ for all $x \in C_k$,
    \begin{multline*}
        \rho_{\qq}(M_\alpha f) \geq \sum_{k=1}^\infty \int_{C_k} M_\alpha f(x)^{q(x)} \, dx \\
        \geq \sum_{k=1}^\infty \int_{C_k} \left(\frac{3^{\alpha-1}}{2} k^{-\frac{3\alpha(1-\alpha)}{1+\alpha}}\right)^{\frac{1+\alpha}{3\alpha(1-\alpha)}} \, dx \geq \left(\frac{3^{\alpha-1}}{2}\right)^{\frac{1+\alpha}{3\alpha(1-\alpha)}} \sum_{k=1}^\infty k^{-1} = \infty.
    \end{multline*}
    Therefore, again by \cite[Proposition 2.12]{VLS}, $M_\alpha f \not \in L^{\qq}(\R)$. \\
    \end{proof}

 \begin{proof}[Construction of Example~\ref{example:K0stronger}]
    We will construct an exponent function $\pp \in \mathcal{P}(\R)$ such that $\| \chi_Q \|_{\pp} \simeq |Q|^{\frac{1}{p_Q}}$ for each interval $Q$, but for each $j \in \N$ there exists an interval $Q_j$ such that $\| \chi_{Q_j} \|_{p'(\cdot)} \geq j|Q_j|^{1/p_{Q_j}'}$. Then
    \[ \| \chi_{Q_j} \|_{\pp} \| \chi_{Q_j} \|_{p'(\cdot)} \gtrsim j|Q_j|^{1/p_{Q_j}} |Q_j|^{\frac{1}{p'_{Q_j}}} = j|Q|, \]
    so $\pp \not \in K_0(\R)$. Our construction will proceed as follows:  let $\pp$ be constant except for smooth perturbations of height $p_+ - p_-$ and width 1 spreading at a rate of $k^\beta$. There is a lower threshold on $\beta$, namely $\beta = p_+/p_-$, for which $\|\chi_Q\|_{\pp} \simeq |Q|^{\frac{1}{p_Q}}$ remains true. Thus by choosing $p_+$ and $p_-$ such that $p_+/p_- < (p')_+/(p')_-$, we can construct $\pp$ such that the perturbations are too close for $\| \chi_Q \|_{p'(\cdot)} \simeq |Q|^{1/p'_Q}$ to hold but far enough apart so that $\|\chi_Q\|_{\pp} \simeq |Q|^{\frac{1}{p_Q}}$.

   The details of our construction are very similar to the construction of Example~\ref{example:K0alpha-not-suff}.  Let $\varphi \in C_c^{\infty}(\R)$ satisfy $\supp(\varphi) \subset [0,1]$, $0 \leq \varphi(x) \leq 4/5$, and $\varphi(x) = 4/5$ for $x \in [1/4,3/4]$. Define
    \[ p(x) = \frac65 + \sum_{k=1}^\infty \varphi(x - k^{2}). \]
  Then we have that $6/5 \leq p(x) \leq 2$, $2 \leq p'(x) \leq 6$, and $\pp, p'(\cdot) \in LH_0(\R)$. We will first show that $\| \chi_Q \|_{\pp} \simeq |Q|^{\frac{1}{p_Q}}$. Let $\E = \{Q : |Q| \leq 1\}$. Since $\pp,p'(\cdot) \in LH_0(\R)$, if we argue as we did for inequalities~\eqref{eqn:smallcube1} and~\eqref{eqn:smallcube2}, we get $\pp \in K_0(\E)$.  But then, by Proposition~\ref{norm and harmonic mean proposition}, we get that $\| \chi_Q \|_{\pp} \leq M_1 |Q|^{\frac{1}{p_Q}}$ for some $M_1 > 0$ and all intervals $Q$, $|Q| \leq 1$. 
  
 Now suppose that $|Q|\geq 1$.  Then for some $j \geq 2$, $(j-1)^2 \leq |Q| \leq j^2$. For each $k$, let  $A_k := [k^2, k^2+1]$; then $Q$ intersects at most $j_0 \leq j$ of the $A_k$, say $A_{k_1},\dots,A_{k_{j_0}}$. Define $P := Q \setminus \bigcup_{i=1}^{j_0} A_{k_i}$; then $|P| \geq (j-1)^2 - j_0 = j^2 - 2j - j_0 + 1$ and $|Q \setminus P| \geq j_0 - 2$.  Given this, we have that
    \begin{align*}
        \frac{1}{p_Q} = \dashint_Q \frac{1}{p(x)} \, dx &= |Q|^{-1} \left(\int_{P} \frac{1}{p(x)} \, dx + \int_{Q \setminus P} \frac{1}{p(x)} \, dx \right) \\
        &\geq |Q|^{-1}\left(\frac{5|P|}{6} + \frac{|Q \setminus P|}2\right) \\
        &\geq |Q|^{-1}\left(\frac{5}{6}(j^2-2j-j_0+1) + \frac{j_0-2}{2}\right) \\
        &\geq j^{-2}\left(\frac{5}{6}j^2 - 2j - \frac{1}{6}\right). 
    \end{align*}
    Let $M_2 \geq 1$ be a constant whose value will be fixed below. If we estimate the modular, we get
    \begin{multline*}
        \rho_{\pp}(\chi_Q / M_2|Q|^{\frac{1}{p_Q}}) = \int_{Q} (M_2|Q|^{\frac{1}{p_Q}})^{-p(x)} \, dx \\
        \leq M_2^{-p_-} \int_{Q} |Q|^{-\frac{p_-}{p_Q}} \, dx 
        \leq M_2^{-1} j^{2(1 - \frac{6}{5}j^{-2}(\frac56 j^2 - 2j - \frac16))},
    \end{multline*}
    where  we have used that  $|Q| \leq j^2$ and our lower bound on $\frac{1}{p_Q}$. We claim that the function $f(x) = x^{2(1 - \frac{6}{5}x^{-2}(\frac56 x^2 - 2x - \frac16))}$ is bounded for $x \geq 1$. It is differentiable for all such $x$ and 
    \[ f'(x) = -\frac{(24x+4)\ln(x) - 24x - 2}{5x^{\frac{5 x^2 - 24 x - 2}{5x^2}}}. \]
    The denominator is positive for $x \geq 1$, and for $x \geq e$, 
    \[ (24x + 4)\ln(x) - 24x - 2 > (24x + 4) - 24x - 2 > 0, \] 
    so that the numerator is positive as well. Thus $f$ is decreasing for $x \in [e,\infty)$. Moreover, $f$ is continuous for $x \geq 1$ so attains a maximum on $[1,e]$. Therefore,  $f$ is bounded for $x \geq 1$. Let $M_2 = \max_{x \geq 1} f(x)$; then 
    \[ \rho_{\pp}(\chi_Q / M_2|Q|^{\frac{1}{p_Q}}) \leq M_2^{-1} f(j) \leq 1. \]
    If we let $M := \max\{M_1,M_2\}$, then we have that  $\| \chi_Q \|_{\pp} \leq M|Q|^{\frac{1}{p_Q}}$ for all intervals $Q$, $|Q|\geq 1$. The opposite inequality holds for all $Q$:  see the proof of inequality~\eqref{|E| 1/P_E Relation}. Therefore, we have shown that $\| \chi_Q \|_{\pp} \simeq |Q|^{\frac{1}{p_Q}}$ for all cubes $Q$.

    On the other hand, fix $j \geq 2$ and let $Q_j = [1,j^{30}+1]$. Define $C_k = [k^2 + \frac14,k^2+\frac34]$ and let $B_j := \bigcup_{k=1}^{j^{15}} C_k$; then  $B_j  \subset Q_j$ and $p'(x) = 2$ for $x \in B_j$. We claim that $p'_{Q_j} \geq 5$. To see this, note that this is equivalent to $\frac{1}{p_{Q_j}}\geq \frac{4}{5}$.   To see this. let $\bar{C}_k = [k^2,k^2+1]$ and let $\bar{B}_j= \bigcup_{k=1}^{j^{15}} \bar{C}_k$.  Then on $Q_j\setminus \bar{B}_j$, $p(x)=\frac{6}{5}$, and so
    \[ \frac{1}{p_{Q_j}} \geq \frac{1}{|Q_j|} \int_{Q_j \setminus \bar{B}_j} \frac{1}{p(x)}\,dx 
   =\frac{5}{6}\frac{|Q_j \setminus \bar{B}_j|}{|Q_j|} > \frac{4}{5}. 
\]
Therefore, 
    \begin{equation*}
        \rho_{p'(\cdot)}(\chi_Q/j|Q_j|^{\frac{1}{p'_{Q_j}}}) 
        \geq \int_{B_j} (j|Q_j|^{\frac{1}{p_{Q_j}}})^{-p'(x)} \, dx 
       \geq j^{-2} |B_j| |Q_j|^{-\frac{2}{5}} = \frac j2 \geq 1.
    \end{equation*}
    Hence, $\| \chi_{Q_j} \|_{\cpp} \geq j|Q_j|^{\frac{1}{p'_{Q_j}}}$, and so $\pp \not\in K_0(\R)$.
    \end{proof}

\medskip

    \begin{proof}[Construction of Example~\ref{example:ppK0-not K0alpha}]
        If $\pp \in K_0^\alpha(\R)$, then by Lemma~\ref{K_0^alpha implies K_0 lemma}, $\pp,\qq \in K_0(\R)$. Thus it suffices to exhibit an exponent function such that $\pp \in K_0(\R)$ but $\qq \not \in K_0(\R)$. We will construct $\pp$ and $\qq$ by modifying the construction in Example~\ref{example:K0stronger}.  Choose constants $1 < p_- < \frac1\alpha$ and $p_- < p_+ < \frac1\alpha$. If we define $1/q_\pm = 1/p_{\pm} - \alpha$, then
        \[ \frac{q_+}{q_-} = \frac{p_+}{p_-} \left(\frac{1-\alpha p_-}{1-\alpha p_+} \right) > \frac{p_+}{p_-}. \]
       Let $\beta_p = p_+ / p_-, \, \beta_q = q_+/q_-$, and fix $\delta > 1$ be such that $\beta_q = \delta \beta_p$. Let $\varphi \in C_c^{\infty}(\R)$ satisfy $\supp(\varphi) \subset [0,1]$, $0 \leq \varphi(x) \leq p_+ - p_-$, and $\varphi(x) = p_+ - p_-$ for $x \in [1/4,3/4]$. Define
        \[ p(x) = p_+ - \sum_{k=1}^\infty \varphi(x - k^{\beta_p}). \]
        We claim $\pp \in K_0(\R)$. It is immediate that $\pp,p'(\cdot) \in LH_0(\R)$. Arguing as in Example~\ref{example:K0stronger}, tt follows that for intervals $Q$, $|Q| \leq 1$, there exists a constant $C > 0$ such that
        \[ \lVert \chi_Q \rVert_{\pp} \lVert \chi_Q \rVert_{p'(\cdot)} \leq C|Q|. \]

        Now suppose $|Q|>1$.  Then  for some $j \geq 2$, $(j-1)^{\beta_p} \leq |Q| \leq j^{\beta_p}$. Define $A_k = [k^{\beta_p}, k^{\beta_p}+1]$ for each $k \geq 1$.  The interval $Q$ intersects at most $j_0 \leq j$ consecutive intervals $A_k$, say $A_{k_1},\dots,A_{k_{j_0}}$. Define $P = Q \setminus \bigcup_{i=1}^{j_0} A_{k_i}$.  Then  by \cite[Corollary~2.23]{VLS},
        \[ \rho_\pp\big(\chi_{\bigcup_{i=1}^{j_0} A_{k_i}}\big) = j_0,  \]
        and using  that $p(x) = p_+$ for $x \in P$ and $j^{1/p_-} \leq 2(j-1)^{1/p_-} \leq 2|Q|^{1/p_+}$, we have that
    \begin{multline*}
         \lVert \chi_Q \rVert_{\pp} \leq \lVert \chi_{\bigcup_{i=1}^{j_0} A_{k_i}} \rVert_{\pp} + \lVert \chi_P \rVert_{\pp} \\
         \leq j^{1/p_-} + |P|^{1/p_+} \leq 2(j-1)^{1/p_-} + |Q|^{1/p_+} \leq 3|Q|^{1/p_+}.
    \end{multline*}
    Similarly, using that  $p'(x) = (p')_-$ for $x \in P$ and $j \leq 2(j-1)^{\beta_p} \leq 2|Q|$,  we have that
    \[ \lVert \chi_Q \rVert_{p'(\cdot)} \leq j^{1/(p')_-} + |P|^{1/(p')_-} \leq (2|Q|)^{1/(p')_-} + |Q|^{1/(p')_-} \leq 3|Q|^{1/(p')_-}.\]
    If we multiply these two estimates, we get
    \[ \lVert \chi_Q \rVert_{\pp} \lVert \chi_Q \rVert_{p'(\cdot)} \leq (3|Q|^{1/p_+})(3|Q|^{1/(p')_-}) = 9|Q|. \]
    Thus $\pp \in K_0(\R)$. \\

    \medskip
    
    We will now show that $\qq \not \in K_0(\R)$ by showing that for all $j$ sufficiently large, there exists an interval $Q_j$ such that $\| \chi_{Q_j} \|_{\qq} \geq j|Q_j|^{1/q_{Q_j}}$. Let 
    \[ \gamma = \frac{1+q_-}{1-\frac{2}{\delta+1}}. \]
    For each $j \geq 1$, define $Q_j = [1,j^{\beta_p \gamma} + 1]$. For all $k$, let  $C_k = [k^{\beta_p} + \frac14, k^{\beta_p} + \frac34]$. Then $B_j := \bigcup_{k=1}^{\lfloor j^\gamma \rfloor} C_k \subset Q_j$ and $q(x) = q_-$ for $x \in B_j$.  Since $q(x)=q_+$ on $Q_j\setminus B_j$, it follows that $q_{Q_j} \to q_+$ as $j \to \infty$.  Hence, there exists $J\geq 2$  such that if $j \geq J$,
    \[ q_{Q_j} \geq \frac{q_+ + \beta_p q_-}{2}. \]
    Therefore, for $j \geq J$, 
    \begin{align*}
        \rho_{\qq}(\chi_Q / j|Q_j|^{1/q_{Q_j}}) &\geq \int_{B_j} (j|Q_j|^{1/q_{Q_j}})^{-q(x)} \, dx \\
        &\geq j^{-q_-}|B_j||Q_j|^{-q_-/q_{Q_j}} = \frac{1}{2}j^{\gamma - q_- - \frac{\beta_p \gamma q_-}{q_{Q_j}}}.
    \end{align*}
    By our bound on $q_{Q_j}$ and the definition of $\gamma$,
    \[ \gamma - q_- - \frac{\beta_p \gamma q_-}{q_{Q_j}} \geq \gamma - q_- - \frac{2\beta_p \gamma q_-}{q_+ + \beta_p q_-} = \gamma - q_- - \frac{2\gamma}{\delta+1} = 1. \]
    Thus $\rho_{\qq}(\chi_Q / j|Q_j|^{1/q_{Q_j}}) \geq 1$ for $j \geq J$. Therefore, $\|\chi_Q \|_{\qq} \geq j|Q|^{1/q_{Q_j}}$, and our proof is complete.
    \end{proof}

    \begin{proof}[Construction of Example~\ref{example:qqK0-notK0alpha}]
        The construction is nearly identical to the construction of Example~\ref{example:ppK0-not K0alpha}. Choose $1 < p_- < \frac1\alpha$ and $p_- < p_+ < \frac1\alpha$. Define
        \[ \beta_{p'} := \frac{(p')_+}{(p')_-} = \frac{p_-(p_+ - 1)}{p_+(p_- - 1)}, \hspace{0.8cm} \beta_{q'} := \frac{(q')_+}{(q')_-} = \frac{p_-(p_+ - 1 + \alpha p_+)}{p_+(p_- - 1 + \alpha p_-)}.\]
        Since $p_+ > p_-$, we have that
        \[ (p_+ - 1)(p_- - 1 + \alpha p_-) > (p_- - 1)(p_+ - 1 + \alpha p_+), \]
        and so $\beta_{p'} > \beta_{q'}$. Let $\varphi \in C_c^{\infty}(\R)$ satisfy $\supp(\varphi) \subset [0,1]$, $0 \leq \varphi(x) \leq p_+ - p_-$, and $\varphi(x) = p_+ - p_-$ for $x \in [1/4,3/4]$. Define
        \[ p(x) = p_- + \sum_{k=1}^\infty \varphi(x - k^{\beta_{q'}}). \]
        We claim $\qq \in K_0(\R)$; the proof is the same as the proof that $\pp \in K_0(\R)$ in Example~\ref{example:ppK0-not K0alpha}.  However, $\pp \not \in K_0(\R)$. One can again repeat the previous argument, using the fact that $\beta_{p'} > \beta_{q'}$,  to show that for each $j$ sufficiently large there exists an interval $Q_j$ such that $\| \chi_{Q_j} \|_{p'(\cdot)} \geq j|Q_j|^{\frac{1}{p'_{Q_j}}}$. Therefore, we omit the details.
    \end{proof}




%
\bibliographystyle{plain}
\bibliography{frac-necc-var}

\end{document}